\documentclass[a4paper,10pt]{amsart}
\usepackage{txfonts,dsfont,amsfonts}
\usepackage{mathrsfs}
\usepackage{amsmath,amssymb,amscd,bbm,amsthm}
\usepackage{xy}
\input xy
\input xypic
\xyoption{all}

\numberwithin{equation}{section}

\theoremstyle{plain}
\newtheorem{theorem}{Theorem}[section]

\newtheorem{lemma}[theorem]{Lemma}

\newtheorem{proposition}[theorem]{Proposition}

\theoremstyle{definition}
\newtheorem{definition}[theorem]{Definition}

\newtheorem{example}[theorem]{Example}

\theoremstyle{remark}
\newtheorem*{remark}{Remark}
\newtheorem*{question}{Question}

\newcommand{\N}{\mathbb{N}}
\newcommand{\Z}{\mathbb{Z}}

\DeclareMathOperator{\id}{\rm id }
\DeclareMathOperator{\im}{\rm im}
\DeclareMathOperator{\coker}{\rm coker}

\DeclareMathOperator{\gkdim}{\rm GKdim}

\DeclareMathOperator{\Mod}{\rm Mod}
\DeclareMathOperator{\smod}{\rm mod}
\DeclareMathOperator{\Hom}{\rm Hom}

\DeclareMathOperator{\Aut}{\rm Aut}
\DeclareMathOperator{\Ext}{\rm Ext}

\DeclareMathOperator{\pdim}{\rm pdim}
\DeclareMathOperator{\idim}{\rm idim}
\DeclareMathOperator{\grad}{\rm grad}
\DeclareMathOperator{\gldim}{\rm gldim}
\DeclareMathOperator{\supp}{\rm supp}

\DeclareMathOperator{\Span}{\rm Span}

\DeclareMathOperator{\uHom}{\underline{\rm Hom}}

\DeclareMathOperator{\uExt}{\underline{\rm Ext}}

\DeclareMathOperator{\G}{\rm G}

\begin{document}

\title{Behavior of the Auslander condition with respect to regradings}

\author{G.-S. Zhou,\; Y. Shen\; and \; D.-M. Lu\ }


\address{\rm Zhou \newline \indent
Ningbo Institute of Technology, Zhejiang university, Ningbo 315100, China
\newline \indent E-mail: 10906045@zju.edu.cn
\newline\newline
\indent Shen \newline
\indent Department of Mathematics, Zhejiang Sci-Tech University, Hangzhou 310018, China \newline \indent E-mail: yuanshen@zstu.edu.cn
\newline\newline
\indent Lu \newline
\indent Department of Mathematics, Zhejiang University, Hangzhou 310027, China
\newline \indent E-mail: dmlu@zju.edu.cn}

\begin{abstract}
We show that a noetherian ring graded by an abelian group of finite rank satisfies the Auslander condition if and only if it satisfies the graded Auslander condition. In addition, we also study the injective dimension, the global dimension and the Cohen-Macaulay property from the same perspective of that for the Auslander condtion. A key step of our approach is to establish homological relations between a graded ring $R$, its quotient ring modulo the ideal $\hbar R$ and its localization ring with respect to the Ore set $\{\, \hbar^i\, \}_{i\geq0}$, where $\hbar$ is a homogeneous regular normal non-invertible element of $R$.
\end{abstract}




\maketitle




\section{Introduction}

\newtheorem{maintheorem}{\bf{Theorem}}
\renewcommand{\themaintheorem}{\Alph{maintheorem}}
\newtheorem{mainproposition}[maintheorem]{\bf{Proposition}}

Though a bird's eye view on the graded ring theory  might give the impression that it is nothing but a naive extension of the ordinary ring theory, it has been justified that the grading structure carries substantial information about graded rings and their modules, both in theoretic and computational way. Throughout the paper $G$ and $H$ stand for abelian groups that are of finite rank. For each $G$-graded ring $A$ and each group homomorphism $\varphi:G\to H$, one has an $H$-graded ring $\varphi^*(A)$ with the same underlying ring structure as $A$ and the natural $H$-grading induced by $\varphi$ (see \S \ref{thesubsection-regrading}). This regrading construction gives rise to the following general problem in the graded ring theory: consider the categories of $G$-graded  (left) $A$-modules and of $H$-graded $\varphi^*(A)$-modules, then which properties of one category can be transferred to another. In particular, in the extreme case when $H$ is the trivial group, this problem reduces to establish connections between various properties in the graded setting and those in the ungraded setting. We refer to \cite{Ek,Ha,LV,NV1,NV2,RZ,SoZ} for relevant works in this perspective.

The Auslander condition is a fundamental homological property on rings that permits one to make effective use of homological techniques in noncommutative ring theory. It is expressed as follows on a ring $A$: for every finitely generated left or right $A$-module $M$, every integer $i$ and every submodule $N$ of $\Ext_A^i(M,A)$, it follows that $\Ext_{A}^j(N,A)=0$ for every $j<i$. The Auslander condition yields several classes of rings that are of special interest. A noetherian ring satisfies the Auslander condition is called a \emph{Gorenstein ring}. The arguments of \cite[Lemma 4.5]{Ba} show that a commutative noetherian ring is Gorenstein if and only if its localiztion ring at every prime ideal has a finite injective dimension. A Gorenstein ring with finite left and right injective dimensions (resp. finite global dimension) is called an \emph{Auslander-Gorenstein ring} (resp. \emph{Auslander-regular ring}). Clearly, Auslander-regular rings are Auslander-Gorenstein. Weyl algebras, universal enveloping algebras of finite dimensional Lie algebras, the Sklyanin algebras and known examples of Artin-Schelter regular algebras are all Auslander-regular. We refer to \cite{ASZ1,ASZ2,Bj,FGR,Lev} for basic properties of theses classes of rings.

In this paper, we examine the above mentioned problem for the Gorensteinness, the Auslander-Gorensteinness and the Auslander-regularity. In other words, this paper focus on the behavior of these homological properties with respect to regradings. Using the technique of filtrations, the works \cite{Ek,LV} successfully treated the special case that $G=\Z$ and $H=\{0\}$. Unfortunately, this  technique looses its effectiveness in the general case. Instead, a key step of our approach is to establish homological relations among $G$-graded rings $R$, $R/\hbar R$  and $R[\hbar^{-1}]$, where $\hbar$ denotes a homogenous regular normal non-invertible element of $R$. The main result of this paper is the following.


\begin{maintheorem}\nonumber\label{Auslander-regrading}
Let $A$ be a $G$-graded ring and $\varphi:G\to H$ a group homomorphism.
Then
\begin{enumerate}
\item $A$ is $G$-Gorenstein  if and only if $\varphi^*(A)$ is $H$-Gorenstein.
\item $A$ is  $G$-Auslander-Gorenstein if and only if $\varphi^*(A)$ is $H$-Auslander-Gorenstein.
\item $A$ is $G$-Auslander-regular if and only if $\varphi^*(A)$ is $H$-Auslander-regular, provided that the number of elements of the torsion part of $\ker \varphi$ is invertible in $A$.
\end{enumerate}
\end{maintheorem}

The injective and global dimensions are indispensable for our purpose. We also obtain a result on the behavior of  these two invariants with respect to regradings, which has an interest in its own right and asserts as in the following theorem. The reader should compare it with \cite[Proposition 4.1]{SoZ}  and \cite[Corollary 6.3.6, Corollary 6.4.2]{NV2}, where similar results are proved.

\begin{maintheorem}\label{homo-dim-regrading-general}
Let $A$ be a $G$-graded ring and $\varphi:G\to H$ a group homomorphism.
\begin{enumerate}
\item Suppose that $A$ is left $\G$-noetherian. Then for every module $N \in \Mod^GA$, it follows that
$$\idim_A^GN \leq \idim_{\varphi^*(A)}^H \varphi^*_A(N)\leq \idim_A^GN \ + \text{ the rank of } \ker \varphi.$$
\item Suppose that the number of elements of the torsion part of $\ker\varphi$ is invertible in $A$. Then $$\gldim(\Mod^GA) \leq \gldim(\Mod^H \varphi^*(A)) \leq \gldim(\Mod^GA)\ + \text{ the rank of } \ker \varphi.$$
\end{enumerate}
\end{maintheorem}

The paper is organized as follows. In Section 2, we fix basic definitions and especially we introduce the dehomogenization functor. In the next two sections, we establish several homological relations between $G$-graded rings $R$, $R/\hbar R$  and $R[\hbar^{-1}]$. Section 5 is devoted to prove the main results. In Section 6, we turn to check our ideas for the Cohen-Macaulay property, a companion of the Auslander condition. The final section focus on  graded rings endowed with a filtration by homogeneous subgroups.

\subsection*{Conventions}
Given an abelian group $\Omega$, we write $T(\Omega)$ for its torsion part and let $\bar{\Omega} := \Z\times \Omega$; given a $G$-graded ring $A$, we consider the polynomial ring $A[t]$ and the Laurent ring $A[t,t^{-1}]$ to be $\bar{G}$-graded by $\deg(at^i) = (i,\deg(a))$ when $a\in A$ being homogeneous; given an element $b$ of a ring $\Lambda$, we write $\lambda_b$ (resp. $\rho_b$) for the scalar multiplication of $b$ on any left (resp. right) $\Lambda$-module.

\section{Preliminaries}

\subsection{Basic definitions}\hspace{\fill}

Let $A=\bigoplus_{\gamma\in G}A_\gamma$ be a $G$-graded ring. We denote by $\Mod^G A$ the category of $G$-graded left $A$-modules, and by $\smod^GA$  the full subcategory consisting of finitely generated objects. Morphisms in $\Mod^GA$ are left $A$-module homomorphisms preserving degrees. We identify $\Mod^GA^o$ with the category of $G$-graded right $A$-modules, where  $A^o$ is the opposite $G$-graded ring of $A$.


The \emph{support} of a module $M \in\Mod^GA$ is the set $\supp M:= \{\,\gamma\in G\,|\, M_\gamma\neq0\,\}$.  A subgroup  $N$ of $M$ is called \emph{homogeneous} if $N=\bigoplus_{\gamma\in G}N\cap M_\gamma$, or equivalently, $N$ has a set of homogeneous generators. If all homogeneous submodules of $M$ are finitely generated then $M$ is called \emph{$G$-noetherian}. We say that $A$ is left (resp. right) $G$-noetherian if ${}_AA$ (resp. $A_A$) is $G$-noetheiran.

The {\em Jacobson $G$-radical} of  $A$ is denoted by $J^G(A)$. It equals to the intersection of all maximal homogeneous left ideals of $A$. We refer to \cite[Section 2.9]{NV2} for more details. Note the following graded version of the well-known \emph{Nakayama Lemma}: for any homogeneous ideal $I$ of $A$, it follows that $I\subseteq J^G(A)$ if and only if $IM\ne M$ for each nonzero module $M\in\smod^GA$.

We denote by $\Aut_G(A)$  the group of automorphisms of $A$ that preserve degrees. For a module $M\in\Mod^GA$ (resp. $M\in \Mod^GA^o$) and an automorphism $\mu\in \Aut_G(A)$, we define a new module ${}^\mu M \in \Mod^GA$ (resp. $M^\mu\in \Mod^GA^o$) as follows. It has the same underling abelian groups and $G$-grading as $M$; the left action (resp. right action) is given by $a*m=\mu(a)\cdot m$ (resp. $m*a=m\cdot \mu(a)$).

For each $\gamma\in G$, the \emph{$\gamma$-shift functor} $\Sigma_\gamma:\Mod^GA\to \Mod^GA$ sends an object $M$ to $M(\gamma)$, which equals to $M$ as a left $A$-module but with $G$-grading given by $M(\gamma)_{\gamma'}= M_{\gamma+\gamma'}$. The functor $\Sigma_\gamma$ acts as the identity map on morphisms.  For $M,\, N\in \Mod^GA$, we write $$\uHom_{A}^G(M,N):=\bigoplus\nolimits_{\gamma\in G} \Hom_{\Mod^GA}(M,\Sigma_\gamma (N)). $$
It is standard to check that $\Mod^GA$ is an abelian category with enough projective and injective objects. The projective and injective dimensions of a module $M\in \Mod^GA$ is denoted by $\pdim_A^GM$ and $\idim_A^GM$ respectively. For $i\in \Z$ and $M,\, N\in\Mod^GA$, we write
$$\uExt_A^{i,G}(M,N):=\bigoplus\nolimits_{\gamma\in G}\Ext_{\Mod^GA}^i(M,\Sigma_\gamma(N)).$$
Note that $\uHom_A^G(M,N)$ and $\uExt_A^{i,G}(M,N)$ become $G$-graded left (resp. right) $A$-modules if $M$ (resp. $N$) is a $G$-graded $A$-bimodule. Similar remarks applies to the tensor product $L\otimes_AM$ of a $G$-graded right $A$-module $L$ and a $G$-graded left $A$-module $M$. Here, $L\otimes_AM$ is graded by $\deg  (l\otimes m) = \deg l +\deg m$ when $l\in L$ and $m\in M$ being homogeneous.

The \emph{grade} of a module  $M\in\Mod^GA$  is a useful homological invariant, especially in the discussion of the Auslander condition. It is defined to be the number
$$
\grad_A^GM:=\inf\{\, i\in \Z \, |\,  \uExt_A^{i,G}(M,A)\neq 0\, \} \in \N\cup \{+\infty\}.
$$
Note that $\grad_A^GM$ is bounded from upper by $\min\{\, \pdim_A^GM,\, \idim_{A^o}^GA\, \}$ when $A$ is left $G$-noetherian and $M$ is finitely generated and  nonzero. Also, by definition, $\grad_A^G0=+\infty$.

The following lemma is an easy consequence of the long exact $\uExt$-sequence.

\begin{lemma}\label{grade-short-exact-sequence}
Let $A$ be a $G$-graded ring and $0\to L\to M \to N\to 0$ an exact sequence in $\Mod^GA$. Then $\grad_A^GM\geq \min\{\ \grad_A^G L,\ \grad_A^GN\ \}$  and $\grad_A^GN\geq \min\{\ \grad_A^GL,\ \grad_A^GM\ \}.$ \hspace{\fill} $\Box$
\end{lemma}

Next, we introduce the graded version of the Auslander condition. To save spaces, we denote by $\smod^G_{au}A$ the full subcategory of $\smod^GA$ consisting of all modules  $M$ that satisfy the following condition:
for any $i\geq0$ and any homogeneous $A^o$-submodule $N$ of $\uExt_A^{i,G}(M,A)$, it follows that $\uExt_{A^o}^{j,G}(N,A) =0$ for every $j<i$ (or equivalently $\grad_{A^o}^GN\geq i$). Then  the graded version of the Auslander condition on $A$ just means that $\smod^G_{au}A=\smod^GA$ and $\smod^G_{au}A^o=\smod^GA^o$.

\begin{definition}\label{Auslander-definition}
We say that a $G$-graded ring $A$ is
\begin{enumerate}
\item \emph{$G$-Gorenstein} if it is left and right $G$-noetherian,  $\smod^G_{au}A=\smod^GA$ and $\smod^G_{au}A^o=\smod^GA^o$;
\item \emph{$G$-Auslander-Gorenstein} if it is $G$-Gorenstein and the injective dimension of $A$ in $\Mod^GA$ and $\Mod^GA^o$ are both finite;
\item  \emph{$G$-Auslander-regular} if it is $G$-Gorenstein and the global dimension of $\Mod^GA$ and $\Mod^GA^o$ are both finite.
\end{enumerate}
\end{definition}

We derive readily from Lemma \ref{grade-short-exact-sequence} a useful property of the subcategory $\smod^G_{au}A$ as below.

\begin{lemma}\label{Auslander-extension}
Let $A$ be a $G$-graded ring and $0\to L\to M \to N\to 0$ an exact sequence in $\Mod^GA$. Then, if $L$ and $N$ are in $\smod^G_{au}A$, it follows that $M$ is too. \hspace{\fill} $\Box$
\end{lemma}

\subsection{The regrading functors} \hspace{\fill}

\label{thesubsection-regrading}

Let $A$ be a $G$-graded ring and $\varphi:G\to H$ a group homomorphism. We denote by $\varphi^*(A)$ the $H$-graded ring with the same underlying ring as  $A$ and the  $H$-grading given by $$\varphi^*(A) :=\bigoplus\nolimits_{\delta\in H} \bigoplus\nolimits_{\gamma \in \varphi^{-1}(\delta)} A_\gamma.$$
There are three natural functors between $\Mod^GA$ and $\Mod^H\varphi^*(A)$ induced by $\varphi$, which we will describe below. The notation we take are different from that in \cite{RZ}.
\begin{itemize}\item The \emph{lower star functor}, $\varphi_*^A: \Mod^H\varphi^*(A) \to \Mod^GA$. For any  $N\in \Mod^H\varphi^*(A)$,
    $$
    \varphi_*^A(N) = \bigoplus\nolimits_{\gamma\in G} N_{\varphi(\gamma)} u_\gamma,
    $$
    where $u_\gamma$ is  a placeholder. The action of $a\in A_{\alpha}$ on $n u_\gamma\in \varphi_*^A(N)_\gamma$ yields $(an) u_{\alpha+\gamma}\in \varphi_*^A(N)_{\alpha+\gamma}$.
\item The \emph{upper star functor}, $\varphi^*_A: \Mod^GA\to \Mod^H\varphi^*(A)$. For any  $M\in \Mod^GA$,
    $$\varphi^*_A(M)=\bigoplus\nolimits_{\delta\in H} \bigoplus\nolimits_{\gamma \in \varphi^{-1}(\delta)} M_\gamma.$$
    The underlying $A$-module structure of  $\varphi^*_A(M)$ is the same as that of $M$.
\item The \emph{upper shriek functor}, $\varphi^!_A: \Mod^GA\to \Mod^H\varphi^*(A)$. For any  $M\in \Mod^GA$,
    $$
    \varphi^!_A(M)=\bigoplus\nolimits_{\delta\in H}\prod\nolimits_{\gamma\in\varphi^{-1}(\delta)} M_\gamma.
    $$
    The action of  $a\in A_\alpha$ on $(m_{\gamma})_{\gamma\in \varphi^{-1}(\delta)} \in \varphi^!_A(M)_{\delta}$ yields $(a m_{\gamma- \alpha})_{\gamma\in \alpha+\varphi^{-1}(\delta)} \in \varphi^!_A(M)_{\varphi(\alpha)+\delta}$.
\end{itemize}
These three functors act in the obvious way on morphisms.

Clearly, the regrading functors are all exact. An interesting fact is that $\varphi^*_A$ is left adjoint to $\varphi_*^A$ and  $\varphi^!_A$ is right adjoint to $\varphi_*^A$, see \cite[Proposition 1.3.2]{RZ}. Also, for a module $M\in \Mod^GA$, one has
$$\varphi^A_*(\varphi_A^*(M)) \cong \bigoplus\nolimits_{\gamma\in \ker \varphi} \Sigma_{\gamma} M, \quad \text{and} \quad  \varphi^A_*(\varphi_A^!(M)) \cong \prod\nolimits_{\gamma\in \ker \varphi} \Sigma_{\gamma} M.$$
Furthermore, the natural inclusion morphism $\varphi_A^*(M) \to \varphi_A^!(M)$ is an isomorphism if and only if $M$ is \emph{$\varphi$-finite}, that is, $\supp M \cap \varphi^{-1}(\delta)$ is finite for all $\delta\in H$.

The above observations readily yields the following result.

\begin{proposition}\label{homo-dim-regrading-basic}
Let $A$ be a $G$-graded ring and $\varphi:G\to H$ a group homomorphism. Then
\begin{enumerate}
\item $\pdim_{\varphi^*(A)}^H \varphi^*_A(M)= \pdim_A^GM$  for every module $M\in \Mod^GA$.
\item $\idim_{\varphi^*(A)}^H \varphi^*_A(M) \geq \idim_A^GM$ for every module  $M\in \Mod^GA$.
\item  $\idim_{\varphi^*(A)}^H \varphi^*_A(M) = \idim_A^GM$ for every $\varphi$-finite module $M\in \Mod^GA$.
\item $\gldim(\Mod^GA)\leq \gldim(\Mod^H\varphi^*(A))$. \hspace{\fill} $\Box$
\end{enumerate}
\end{proposition}

Recall that a module $M\in \Mod^GA$ is said to be \emph{pseudo-coherent} if it admits a projective resolution in $\Mod^GA$ with each term finitely generated. Clearly, if $A$ is left $G$-noetherian then pseudo-coherent objects of $\Mod^GA$ coincide with finitely generated objects.

\begin{lemma}\label{ext-regrading-basic}
Let $A$  be a $G$-graded ring and  $\varphi:G\to H$  a group homomorphism. Let $M\in\Mod^GA$ be a pseudo-coherent module. Then there is a natural isomorphism
$$
\uExt_{\varphi^*(A)}^{i,H} (\varphi_A^*(M), \varphi^*(A)) \cong\varphi_{A^o}^*(\uExt_A^{i,G}(M,A))
$$
in $\Mod^H\varphi^*(A)^o$ for every integer $i\in \Z$. Consequently, $\grad_A^GM=\grad_{\varphi^*(A)}^HM$. \hspace{\fill} $\Box$
\end{lemma}

\begin{lemma}\label{ext-regrading-injective}
Let $A$ be a $G$-graded ring and $\varphi:G\to H$ an injective group homomorphism. Then the functor $\varphi_A^*:\Mod^GA\to \Mod^H\varphi^*(A)$ is fully faithful and every module in $\Mod^H\varphi^*(A)$ is a direct sum of shifts of modules in the image of $\varphi_A^*$. Moreover, there is a natural isomorphism
$$\uExt_{\varphi^*(A)}^{i,H}(\varphi^*_A(M), \varphi^*(A)) \cong \varphi_{A^o}^*(\uExt_A^{i,G}(M,A))$$
in $\Mod^H\varphi^*(A)^o$ for every integer $i\in \Z$ and every module $M\in\Mod^GA$.  \hspace{\fill} $\Box$
\end{lemma}

\subsection{The dehomogenization functor} \hspace{\fill}

Let $A$ be a $G$-graded ring and $\varphi:G\to H$ a group homomorphism. For a subsemigroup $\Omega$ of $G$, the semigroup ring $A[\Omega]$ is the free $A$-module $\bigoplus_{\delta\in \Omega} A e_\delta$ equipped with the multiplication given by $ae_\alpha\cdot be_\beta=(ab) e_{\alpha+\beta}$ for $a,b\in A$ and $\alpha,\beta \in \Omega$. We introduce a $G$-grading on $A[\Omega]$ by putting $$A[\Omega]_\gamma=\bigoplus\nolimits_{\delta\in \Omega} A_{\gamma-\delta}e_\delta, \quad \gamma\in G.$$
Look at the maps
$$
A\xrightarrow{\iota} A[\Omega] \xrightarrow{\rho} A
$$
defined by $\iota(a)= ae_0$ and $\rho(\sum_{\delta\in \Omega} a_{\delta}e_\delta) =\sum_{\delta\in \Omega} a_{\delta}$. Clearly, $\rho\circ \iota=\id_A$ and $\ker(\rho)=\sum_{\delta,\delta'\in \Omega} A\cdot(e_\delta-e_{\delta'})$. Note that  $\iota$ is a homomorphism of $G$-graded rings but $\rho$ is not unless $\Omega$ is trivial. However, if $\Omega\subseteq \ker \varphi$ then $\rho: \varphi^* (A[\Omega]) \to \varphi^*(A)$ is a homomorphism of $H$-graded rings.


The \emph{dehomogenization functor} $\Xi_A^\varphi:\Mod^GA[\ker\varphi]\to \Mod^H\varphi^*(A)$ is defined to be the composition of functors $\Mod^G A[\ker\varphi]  \xrightarrow{\varphi^*_{A[\ker\varphi]}} \Mod^H\varphi^*(A[\ker\varphi]) \xrightarrow{\varphi^*(A)\otimes_{\varphi^*(A[\ker\varphi])}-} \Mod^H\varphi^*(A).$
It is easy to check that, up to natural isomorphisms, there is a commutative diagram of abelian categories
$$
\xymatrix{
& \Mod^GA \ar@/_/[dl]|-{\iota^*} \ar@/^/[dl]|-{\iota^!} \ar@/^/[dr]|-{\varphi^*_A} \ar@/_/[dr]|-{\varphi^!_A}  & \\
\Mod^G A[\ker\varphi]  \ar[rr]|-{\Xi_A^\varphi} \ar[ur]|-{\iota_*}  & & \Mod^H\varphi^*(A) \ar[ul]|-{\varphi_*^A},
}
$$
where $\iota_*$ is the scalar restriction functor, $\iota^* = A[\ker\varphi] \otimes_A-$ and $\iota^!=\uHom_{A}^G(A[\ker\varphi],-)$.

\begin{lemma}\label{ext-regrading-surjective}
Let $A$ be a $G$-graded ring and $\varphi:G\to H$ a surjective group homomorphism. Then the dehomogenization functor $\Xi_A^\varphi:\Mod^GA[\ker\varphi]\to \Mod^H\varphi^*(A)$ is an equivalence. Moreover,
$$\uExt_{\varphi^*(A)}^{i,H}(\Xi_A^\varphi(M), \varphi^*(A)) \cong \Xi_{A^o}^\varphi(\uExt_{A[\ker\varphi]}^{i,G}(M,A[\ker \varphi]))$$
in $\Mod^H\varphi^*(A)^o$ for every integer $i\in \Z$ and every module $M\in \Mod^GA[\ker \varphi]$.
\end{lemma}

The first statement is proved in \cite[Section 6.4]{NV2}. We give below an alternative demonstration.

\begin{proof}To save notations, let $\Omega:=\ker\varphi$. For each module $M\in \Mod^GA[\Omega]$, let $\nu_M:M\to \Xi_A^\varphi(M)$ be the natural map given by $m\mapsto 1\otimes m.$ We claim that for each $\alpha\in G$ the restriction map $\nu_{M,\alpha}:M_\alpha\to \Xi_A^\varphi(M)_{\varphi(\alpha)}$ is bijective. First note that for each $y =\sum_{\gamma\in \alpha +\Omega} y_\gamma \in \sum_{\gamma\in \alpha+\Omega} M_\gamma$ one has
$$
\sum\nolimits_{\gamma\in \alpha+\Omega} e_{\alpha-\gamma}\cdot y_{\gamma}\in M_\alpha \, \text{ and }\, \nu_M(\sum\nolimits_{\gamma\in \alpha+\Omega} e_{\alpha-\gamma}\cdot y_{\gamma})=\nu_M(y).
$$
It follows that $\nu_{M,\alpha}$ is surjective because
$\nu_M(\sum_{\gamma\in \alpha+\Omega} M_\gamma)=\Xi_A^\varphi(M)_{\varphi(\alpha)}$. Now fix a family $(\gamma_\delta)_{\delta\in H}$ in $G$ with $\varphi(\gamma_\delta)=\delta$ and then define an additive map $f_M:M\to M$ by
$$
y=\sum\nolimits_{\gamma\in G} y_\gamma \, \mapsto \, \sum\nolimits_{\delta\in H} \sum\nolimits_{\gamma\in \gamma_\delta+\Omega} e_{\gamma_\delta-\gamma}\cdot y_\gamma.
$$
It is easy to check that $\nu_M\circ f_M=\nu_M$ and $f_M((e_\beta-e_{\beta'})M_\gamma) =0$ for any $\beta,\, \beta'\in \Omega$ and $\gamma\in G$. Since
$$\ker(\nu_M) = \ker(\rho)\cdot M=\sum\nolimits_{\beta,\,\beta'\in \Omega}(e_\beta-e_{\beta'})\cdot M,$$
$\ker(\nu_M)= \ker(f_M)$ and hence $\ker(\nu_{M,\alpha})=\ker(\nu_M)\cap M_\alpha = \ker(f_M)\cap M_\alpha=0$. Thus  $\nu_{M,\alpha}$ is injective.

Next we turn to show  $\Xi_A^\varphi$ is fully faithful, that is, to show the structure map
$$\Xi_A^\varphi:\Hom_{\Mod^GA[\Omega]}(M,N)\to \Hom_{\Mod^H\varphi^*(A)}(\Xi_A^\varphi(M),\Xi_A^\varphi(N))$$
is bijective for any modules $M,N\in \Mod^GA[\Omega]$. Let $f\in \ker(\Xi_A^\varphi)$. Then $\nu_N\circ f = \Xi_A^\varphi(f)\circ \nu_M = 0$. So $\nu_{N,\alpha}\circ f_\alpha=0$ for each $\alpha\in G$ and hence $f = 0$ (here, $f_\alpha$ denotes the $\alpha$-th piece of $f$). Thererfore $\Xi_A^\varphi$ is injective.  For any $g\in \Hom_{\Mod^H\varphi^*(A)}(\Xi_A^\varphi(M),\Xi_A^\varphi(N))$, define a map $\tilde{g}:M\to N$ by $$y=\sum\nolimits_{\gamma \in G} y_\gamma \, \mapsto \, \sum\nolimits_{\gamma\in G} (\nu_{N,\gamma})^{-1} (g (\nu_M (y_\gamma))).$$  It is easy to check that   $\tilde{g}\in \Hom_{\Mod^GA[\Omega]}(M,N)$ and $\Xi_A^\varphi(\tilde{g})=g$. Thus $\Xi_A^\varphi$ is also surjective.

To see $\Xi_A^\varphi$ is an equivalence, it remains to show every module $M\in \Mod^H\varphi^*(A)$ is isomorphic to some object in the image of $\Xi_A^\varphi$. Clearly, every free module in $\Mod^H\varphi^*(A)$ is isomorphic to the image of some free module in $\Mod^GA[\Omega]$ under $\Xi_A^\varphi$. Since $\Xi_A^\varphi$ is fully faithful, one may choose an exact sequence $\Xi_A^\varphi(L') \xrightarrow{\Xi_A^\varphi(f)} \Xi_A^\varphi(L) \to M\to 0$ in $\Mod^H\varphi^*(A)$ with $L,L'$ being free in $\Mod^GA[\Omega]$. Note that $\Xi_A^\varphi$ is right exact, one arrives at $\Xi_A^\varphi(\coker(f))\cong M$.

Finally, we justify the second statement. Note that $\Xi_A^\varphi(A[\Omega]) \cong \varphi^*(A)$ as $H$-graded $\varphi^*(A)$-bimodules. Fix a family $(\gamma_\delta)_{\delta\in H}$ in $G$ with $\varphi(\gamma_\delta)=\delta$ and then define  for each module $M\in\Mod^GA[\Omega]$ a map  $\omega_M:\Xi_{A^o}^\varphi(\uHom_{A[\Omega]}^{G}(M,A[\Omega])) \to \uHom_{\varphi^*(A)}^{H}(\Xi_A^\varphi(M), \Xi_A^\varphi(A[\Omega]))$  by
$$
\sum\nolimits_{\delta\in H} f_\delta \mapsto \sum\nolimits_{\delta\in H}\Xi_{A}^\varphi (v_{\uHom_{A[\Omega]}^G(M,A[\Omega]),\gamma_\delta}^{-1}(f_\delta)).
$$
Clearly,  $\omega_M$ is an isomorphism in $\Mod^H\varphi^*(A)^o$ and it is natural in $M$. It follows that
$$
\Xi_{A^o}^{\varphi} \circ \uHom_{A[\Omega]}^G(-,A[\Omega]) \cong \uHom_{\varphi^*(A)}^H(-,\varphi^*(A)) \circ \Xi_A^{\varphi}
$$
as functors from $\Mod^GA[\Omega]$ to $\Mod^H\varphi^*(A)^o$. By standard homological algebra, the result follows.
\end{proof}

\section{On injective and global dimensions}
\label{thesection-3}

Throughout, we reserve $R$ to stand for a $G$-graded ring on which a regular normal non-invertible homogeneous element $\hbar$ of degree $\varepsilon$ is specified. We write $\bar{R}:= R/\hbar R$, the quotient ring of $R$, and $\hat{R}:=R[\hbar^{-1}]$, the localization ring of $R$ at the Ore set $\{\, \hbar^i\,\}_{i\geq0}$, and consider them both $G$-graded in the natural way. Also, we define $\tau_\hbar\in \Aut_G(R)$  by $\hbar\cdot a ={{\tau_\hbar}}(a)\cdot\hbar$ for all $a\in R$. This section focus on the relations between the injective dimensions (resp. global dimensions) of $R$, $\bar{R}$ and $\hat{R}$. 

Consider a module $M\in \Mod^GR$. The \emph{$\hbar$-torsion submodule} of $M$ is defined by
$$
T_\hbar(M) := \{\ x\in M\ |\ \hbar^nx=0\ \text{ for some }n\geq0\ \},
$$
which is clear a homogeneous submodule of $M$. To save notations, we also  write $F_\hbar(M) : = M/T_\hbar(M)$. We say $M$ is \emph{$\hbar$-torsionfree} if $T_\hbar(M)=0$. Besides, $M$ is called {\em $\hbar$-discrete} if for any $\gamma\in G$ there is an integer $n\geq0$ such that $(\hbar^{n}\cdot M)_\gamma =0$. Similar definitions apply to modules in $\Mod^GR^o$.

We start by establishing several technical lemmas on $\uExt$-groups.

\begin{lemma}\label{Rees-lemma-general}
Assume $N\in \Mod^GR$ is $\hbar$-torsionfree. Then there is a natural isomorphism
$$
\uExt_{\bar{R}}^{i,G}(L,\bar{R}\otimes_RN) \cong \Sigma_{-\varepsilon}\uExt_R^{i+1,G}({}^{\tau_\hbar} L,N)
$$
in $\Mod^G\Z$ for every integer $i\in \Z$ and every module $L\in \Mod^G\bar{R}$.
\end{lemma}

\begin{proof}
It follows from the collapsing of the following spectral sequence in $\Mod^G\Z$:
$$
\uExt_{\bar{R}}^{p,G}(L,\uExt_R^{q,G}(\bar{R}, \Sigma_{-\varepsilon}{}^{\tau_\hbar^{-1}}N)) \Longrightarrow \uExt_R^{n,G}(L,\Sigma_{-\varepsilon}{}^{\tau_\hbar^{-1}}N)) \cong \Sigma_{-\varepsilon} \uExt_R^{n,G}({}^{\tau_\hbar}L,N).
$$
Note that $\uExt_{R}^{q,G}(\bar{R},\Sigma_{-\varepsilon}{}^{\tau_\hbar^{-1}}N) = 0$ for
$q\neq 1$ and $\uExt_{R}^{1,G}(\bar{R},\Sigma_{-\varepsilon}{}^{\tau_\hbar^{-1}}N) \cong \bar{R}\otimes_RN$ in $\Mod^G\bar{R}$.
\end{proof}

\begin{lemma}\label{long-exact-sequence-general}
Assume $M\in\Mod^GR$ is $\hbar$-torsionfree. Then there is a long exact sequence
$$
\cdots \to \uExt_{R}^{i,G}(\bar{R}\otimes_R M,N)   \to \uExt_R^{i,G}(M,N) \xrightarrow{(\lambda_{\hbar})_* } \uExt_R^{i,G}(M,\Sigma_{\varepsilon}{}^{{{\tau_\hbar}}}N) \to \uExt_{R}^{i+1,G}(\bar{R}\otimes_R M, N)  \to   \cdots
$$
in $\Mod^G\Z$ for every integer $i\in \Z$ and every module $N\in \Mod^GR$.
\end{lemma}

\begin{proof}
This can be read from the following commutative diagram in $\Mod^G\Z$:
{\small$$
\begin{array}{ccccccccccc}
\cdots \!\! & \!\! \to\!\!  & \!\!  \uExt_R^{i,G}(\bar{R}\otimes_R M,N)\!\!  & \!\! \to \!\!  & \!\! \uExt_R^{i,G}(M, N) \!\! & \!\! \xrightarrow{(\lambda_\hbar)^*} \!\!  & \!\!  \uExt_R^{i,G}(\Sigma_{-\varepsilon} {}^{{{\tau_\hbar}}^{-1}}\! M,N) \!\!  & \!\!  \to \!\!  & \!\!  \uExt_R^{i+1,G}(\bar{R}\otimes_R M,N) \!\! & \!\!  \to \!\!   & \!\!  \cdots \\
&&\!\! \downarrow = \!\!& &\!\! \downarrow = \!\!&&\!\! \downarrow \cong \!\!&&\!\! \downarrow = \!\!&& \\
\cdots\!\! & \!\! \to\!\!  & \!\!  \uExt_R^{i,G}(\bar{R}\otimes_R M,N)\!\!  & \!\! \to \!\!  & \!\! \uExt_R^{i,G}(M, N) \!\! & \!\! \xrightarrow{(\lambda_\hbar)_*} \!\!  & \!\!  \uExt_R^{i,G}(M,\Sigma_\varepsilon {}^{{{\tau_\hbar}}}\! N) \!\!  & \!\!  \to \!\!  & \!\!  \uExt_R^{i+1,G}(\bar{R}\otimes_R M,N) \!\! & \!\!  \to \!\!   & \!\!  \cdots,
\end{array}
$$}
where the top row is the long $\uExt$-sequence associated to $0 \to  \Sigma_{-\varepsilon} {}^{{{\tau_\hbar}}^{-1}}\! M \xrightarrow{\lambda_\hbar} M  \to  \bar{R}\otimes_R  M  \to 0$.
\end{proof}

\begin{lemma}\label{localization-ext-general}
Assume $R$ is left $G$-noetherian and $M\in\smod^GR$. Then $\uExt_{\hat{R}}^{i,G}(\hat{R}\otimes_RM, \hat{R}\otimes_RN)$ is naturally isomorphic to the colimit of the direct system
$$
\uExt_{R}^{i,G}(M, N)\xrightarrow{(\lambda_{\hbar})_*} \uExt_{R}^{i,G}(M,\Sigma_\varepsilon {}^{{{\tau_\hbar}}}\!N) \xrightarrow{(\lambda_{\hbar})_*} \uExt_{R}^{i,G}(M,\Sigma_{2\varepsilon}{}^{{{\tau_\hbar}}^2}\!N) \to \cdots
$$
in $\Mod^G\Z$ for every integer $i\in \Z$ and every module $N\in \Mod^GR$.
\end{lemma}

\begin{proof}
We have: $\uExt^{i,G}_{\hat{R}}(\hat{R}\otimes_RM, \hat{R}\otimes_RN) \cong \uExt_{R}^{i,G}(M,\hat{R}\otimes_RN)$ in $\Mod^G\Z$; the functor $\uExt_R^{i,G}(M,-)$ preserves colimits; and the colimit of the direct system $N\xrightarrow{\lambda_\hbar} \Sigma_\varepsilon {}^{{{\tau_\hbar}}}\! N \xrightarrow{\lambda_\hbar}  \Sigma_{2\varepsilon} {}^{{{\tau_\hbar}}^2}\! N \to \cdots$ is naturally isomorphic to $\hat{R}\otimes_RN$ in $\Mod^GR$.  The result follows directly.
\end{proof}

\begin{proposition}\label{idim-cut-at-infinity-general}
Assume $R$ is left $G$-noetherian and $N\in \Mod^GR$ is $\hbar$-torsionfree. Then
$$\idim^G_RN = \max\{\  \idim^G_{\bar{R}}\bar{R}\otimes_RN+1,\ \idim^G_{\hat{R}}\hat{R}\otimes_RN\ \}.$$
\end{proposition}

\begin{proof}
Let $d=\max\{\  \idim^G_{\bar{R}}\bar{R}\otimes_RN+1,\ \idim^G_{\hat{R}}\hat{R}\otimes_RN\ \}$. By Lemma \ref{Rees-lemma-general} and \cite[Theorem 1.3]{GJ}, we have $\idim_R^GN\geq d$. Thus we may assume $d<\infty$ and are left to show $\uExt_R^{d+1,G}(M,N)=0$ for every module $M\in \smod^GR$. Note that $T_\hbar(M)$ is finitely generated, so $\hbar^nT_\hbar(M)=0$ for some integer $n\geq1$. By Lemma \ref{Rees-lemma-general}, $\uExt_R^{d+1,G}(\hbar^iT_\hbar(M)/\hbar^{i+1}T_\hbar(M),N) =0$ for every integer $i\geq 0$, and thereof $$\uExt_R^{d+1,G}(T_\hbar(M), N)=0.$$
Lemma \ref{Rees-lemma-general} also tells us that $\uExt_R^{d+1,G}(\bar{R}\otimes_RF_\hbar(M),\Sigma_{i\varepsilon}{}^{{{\tau_\hbar^i}}}N) =  \uExt_R^{d+2,G}(\bar{R}\otimes_RF_\hbar(M), \Sigma_{i\varepsilon}{}^{{{\tau_\hbar^i}}}N) =0$ for every integer $i\geq0$. Then, by Lemma \ref{long-exact-sequence-general} and Lemma \ref{localization-ext-general}, we have $$\uExt_R^{d+1,G}(F_\hbar(M),N)\cong \uExt_{\hat{R}}^{d+1,G}(\hat{R}\otimes_RF_\hbar(M), \hat{R}\otimes_RN) = 0.$$
Now the desired equality $\uExt_R^{d+1,G}(M,N)=0$ follows immediately.
\end{proof}

\begin{proposition}\label{idim-cut-at-infinity-special}
Assume $R$ is left $G$-noetherian and $N\in \Mod^GR$ is $\hbar$-torsionfree and $\hbar$-discrete. Then
$$\idim_{\hat{R}}^G\hat{R}\otimes_RN \leq \idim_{\bar{R}}^G\bar{R}\otimes_R N \quad \text{ and } \quad \idim^G_RN  = \idim_{\bar{R}}^G\bar{R}\otimes_R N+1.$$
\end{proposition}

\begin{proof}
By Proposition \ref{idim-cut-at-infinity-general}, we may assume $\idim^G_{\bar{R}}\bar{R}\otimes_RN =d <\infty$ and then we are left to show $\idim_{\hat{R}}^G\hat{R}\otimes N\leq d$. By Lemma \ref{localization-ext-general}, it suffices to show that $\uExt_R^{d+1,G}(M,\Sigma_{j\varepsilon}{}^{\tau_{\hbar}^j}N)=0$ for every integer $j\geq0$ and every $\hbar$-torsionfree module $M\in\smod^GR$. By Lemma \ref{Rees-lemma-general} and Lemma \ref{long-exact-sequence-general},  $$(\lambda_\hbar^s)_*:\uExt_{R}^{d+1,G}(M,\Sigma_{(j-s)\varepsilon}{}^{\tau_\hbar^{j-s}}N) \to \uExt_{R}^{d+1,G}(M,\Sigma_{j\varepsilon}{}^{\tau_\hbar^{j}}N)$$ is surjective for every integer $s\geq1$. Choose a projective resolution $\cdots \to P_1 \xrightarrow{\partial_1} P_0\to M\to 0$ in $\Mod^GR$ with all $P_n$ finitely generated. Then,  given any $f\in \uHom_R^G(P_{d+1},\Sigma_{j\varepsilon}{}^{\tau_{\hbar}^j}N)_\alpha$ with $f\circ \partial_{d+2}=0$, there is a $g_s\in\uHom_R^G(P_{d},\Sigma_{j\varepsilon}{}^{\tau_{\hbar}^j}N)_\alpha$ and an $h_s\in \uHom_R^G(P_{d+1},\Sigma_{(j-s)\varepsilon}{}^{\tau_{\hbar}^{j-s}}N)_\alpha$ such that $$f-g_s\circ \partial_{d+1}=\lambda_\hbar^s\circ h_s.$$
Since $P_{d+1}$ is finitely generated and $N$ is $\hbar$-torsionfree and $\hbar$-discrete, one concludes  that $\lambda_\hbar^s\circ h_s=0$ and hence $f=g_s\circ \partial_{d+1}$ for $s\gg0$. Consequently, the desired equality $\uExt_R^{d+1,G} (M,\Sigma_{j\varepsilon}{}^{\tau_{\hbar}^j}N) = 0$ holds.
\end{proof}

\begin{proposition}\label{gldim-cut-at-infinity-general}
Assume $R$ is left $G$-noetherian and $\gldim(\Mod^G\bar{R}) <\infty$. Then
$$\gldim(\Mod^GR)=\max\{\ \gldim(\Mod^G\bar{R}) + 1,\ \gldim(\Mod^G\hat{R})\ \}.$$
\end{proposition}

This proposition is a generalization of \cite[Chapter I, Section 7.2, Theorem 4 (2)]{LV}, where $R$ is assumed to be left and right $G$-noetherian and $\hbar$ is regular central.

\begin{proof}
Let $d=\max\{\ \gldim(\Mod^G\bar{R}) + 1,\ \gldim(\Mod^G\hat{R})\ \}$. Then $\gldim(\Mod^GR)\geq d$ by the graded version of \cite[Theorem 7.3.5 (b), Corollary 7.4.3]{MR}. We may assume $d<\infty$ and proceed to see $\gldim(\Mod^GR)\leq d$. It suffices to show $\pdim_R^GM\leq d$ for every module $M\in\smod^GR$. By the graded version of \cite[Theorem 7.3.5 (a)]{MR}, we have $\pdim_R^G \hbar^iT_\hbar(M)/\hbar^{i+1}T_\hbar(M) \leq d$ for every integer $i\geq 0$.
Since $T_\hbar(M)$ is finitely generated,  $\hbar^n T_\hbar(M)=0$ for some integer $n\geq 1$. It follows that $\pdim_R^G T_\hbar(M)\leq d.$
The graded version of \cite[Theorem 7.3.5 (a)]{MR}  also yields $\pdim_{R}^G \bar{R}\otimes_RF_\hbar(M) \leq d$. Then, for every module $N\in \Mod^GR$, Lemma \ref{long-exact-sequence-general} together with Lemma \ref{localization-ext-general} tells us that
$$
\uExt_{R}^{d+1,G}(F_\hbar(M),N) \cong \uExt^{d+1,G}_{\hat{R}}(\hat{R}\otimes_RF_\hbar(M), \hat{R}\otimes_RN) = 0.
$$
Consequently, $\pdim_R^GF_\hbar(M) \leq d.$ Now the desired inequality $\pdim_R^GM \leq d$ follows immediately.
\end{proof}

\begin{proposition}\label{gldim-cut-at-infinity-special}
Assume one of the following two conditions hold: (a) $R$ is left $G$-noetherian and $\hbar\in J^G(R)$; (b) $p(\supp R) \subseteq \N$ and $p(\varepsilon)>0$ for some group homomorphism $p:G\to \Z$. Then
$$
\gldim(\Mod^G\hat{R}) \leq \gldim(\Mod^G\bar{R}).
$$
Furthermore, if (a) holds and $\gldim(\Mod^G\bar{R}) =d <\infty$ then $\gldim(\Mod^GR) =d+1$.
\end{proposition}

\begin{proof}
By the graded version of \cite[Proposition 7.3.6 (b)]{MR} under condition (a) and by an  obvious modification of the discussion of \cite[Proposition 7.3.6 (b)]{MR} under condition (b), we obtain that $$\pdim_{\hat{R}}^G\hat{R}\otimes_RM \leq \pdim_R^GM = \pdim_{\bar{R}}^G\bar{R}\otimes_R M$$ for every $\hbar$-torsionfree module $M\in \smod^GR$. It follows readily that $\gldim(\Mod^G\hat{R}) \leq \gldim(\Mod^G\bar{R})$. The last statement is a direct consequence of this inequality and Proposition \ref{gldim-cut-at-infinity-general}.
\end{proof}

\section{On the Auslander condition}

We follow the notations and conventions in the previous sections. In particular, $R$ is a $G$-graded ring and $\hbar\in R$ is a homogeneous regular normal non-invertible element of degree $\varepsilon$. This section is devoted to study the relations between the Auslander conditions of $R$, $\bar{R}$ and $\hat{R}$.

First we establish several technical lemmas on $\uExt$-groups.

\begin{lemma}(Rees' Lemma)\label{Rees-lemma-regular}
There is a natural isomorphism
$$\uExt_{\bar{R}}^{i,G}(L,\bar{R}) \cong \Sigma_{-\varepsilon}\uExt_R^{i+1,G}(L,R)^{{\tau_\hbar}}$$ in $\Mod^GR^o$ for every integer $i\in \Z$ and every module $L\in \Mod^G\bar{R}$.
\end{lemma}

\begin{proof}
It follows from the collapsing of the following spectral sequence in $\Mod^GR^o$:
$$\uExt_{\bar{R}}^{p,G}(L, \uExt_{R}^{q,G}(\bar{R}, \Sigma_{-\varepsilon}{}^{{{\tau_\hbar}}^{-1}}\!R^1)) \, \Longrightarrow \, \uExt_{R}^{n,G}(L, \Sigma_{-\varepsilon}{}^{{{\tau_\hbar}}^{-1}}\!R^1)\cong \Sigma_{-\varepsilon}\uExt_{R}^{n,G}(L, R\,)^{{{\tau_\hbar}}}.$$
Note that $\uExt_{R}^{q,G}(\bar{R}, \Sigma_{-\varepsilon}{}^{{{\tau_\hbar}}^{-1}}\!R^1)=0$ for $q\neq 1$ and  $\uExt_R^{1,G}(\bar{R}, \Sigma_{-\varepsilon}{}^{{{\tau_\hbar}}^{-1}}\!R^1)\cong \bar{R}$ in $\Mod^GR\otimes R^o$.
\end{proof}

\begin{lemma}\label{long-exact-sequence-regular}
Assume $M\in \Mod^GR$ is $\hbar$-torsionfree. Then there is a long exact sequence in $\Mod^GR^{\rm o}$:
$$
\cdots \to \uExt_{\bar{R}}^{i-1,G}(\bar{R}\otimes_R M, \bar{R})   \to \Sigma_{-\varepsilon}\uExt_R^{i,G}(M,R)^{{{\tau_\hbar}}} \xrightarrow{\rho_{\hbar}} \uExt_R^{i,G}(M,R) \to \uExt_{\bar{R}}^{i,G}(\bar{R}\otimes_R M, \bar{R})  \to   \cdots.
$$
\end{lemma}

\begin{proof}
Consider the following commutative diagram in $\Mod^GR^o$:
{\small$$
\begin{array}{ccccccccccc}
\cdots \!\! & \!\! \to\!\!  & \!\!  \uExt_R^{i,G}(\bar{R}\otimes_R M,R)\!\!  & \!\! \to \!\!  & \!\! \uExt_R^{i,G}(M, R) \!\! & \!\! \xrightarrow{(\lambda_\hbar)^*} \!\!  & \!\!  \uExt_R^{i,G}(\Sigma_{-\varepsilon} {}^{{{\tau_\hbar}}^{-1}}\! M,R) \!\!  & \!\!  \to \!\!  & \!\!  \uExt_R^{i+1,G}(\bar{R}\otimes_R M,R) \!\! & \!\!  \to \!\!   & \!\!  \cdots \\
&&\!\! \downarrow = \!\!& &\!\! \downarrow = \!\!&&\!\! \downarrow \cong \!\!&&\!\! \downarrow = \!\!&& \\
\cdots\!\! & \!\! \to\!\!  & \!\!  \uExt_R^{i,G}(\bar{R}\otimes_R M,R)\!\!  & \!\! \to \!\!  & \!\! \uExt_R^{i,G}(M, R) \!\! & \!\! \xrightarrow{(\lambda_\hbar)_*} \!\!  & \!\!  \uExt_R^{i,G}(M,\Sigma_{\varepsilon} {}^{{{\tau_\hbar}}}\! R) \!\!  & \!\!  \to \!\!  & \!\!  \uExt_R^{i+1,G}(\bar{R}\otimes_R M,R) \!\! & \!\!  \to \!\!   & \!\!  \cdots \\
&&\!\!  \downarrow = \!\! & & \!\!  \downarrow = \!\! && \quad \,\,\, \downarrow({{\tau_\hbar}}^{-1})_*  \!\!&& \!\!  \downarrow =  && \\
\cdots\!\! & \!\! \to\!\!  & \!\!  \uExt_R^{i,G}(\bar{R}\otimes_R M,R)\!\!  & \!\! \to \!\!  & \!\! \uExt_R^{i,G}(M, R) \!\! & \!\! \xrightarrow{(\rho_\hbar)_*} \!\!  & \!\!  \uExt_R^{i,G}(M,\Sigma_{\varepsilon}\, R^{{{\tau_\hbar}}^{-1}}) \!\!  & \!\!  \to \!\!  & \!\!  \uExt_R^{i+1,G}(\bar{R}\otimes_R M,R) \!\! & \!\!  \to \!\!   & \!\!  \cdots\\
&&\!\! \downarrow = \!\!& &\!\! \downarrow = \!\!&&\!\! \downarrow \cong  \!\!&&\!\! \downarrow = \!\!&& \\
\cdots\!\! & \!\! \to\!\!  & \!\!   \uExt_R^{i,G}(\bar{R}\otimes_R M,R)\!\!  & \!\! \to \!\!  & \!\!  \uExt_R^{i,G}(M,R) \!\! & \!\! \xrightarrow{\rho_\hbar} \!\!  & \!\!  \Sigma_{\varepsilon}\uExt_R^{i,G}(M,R)^{{{\tau_\hbar}}^{-1}} \!\!  & \!\!  \to \!\!  & \!\!  \uExt_R^{i+1,G}(\bar{R}\otimes_R M,R) \!\! & \!\!  \to \!\!   & \!\!  \cdots,
\end{array}
$$}

\noindent where the top row is the long $\uExt$-sequence associated to $0 \to  \Sigma_{-\varepsilon} {}^{{{\tau_\hbar}}^{-1}}\! M \xrightarrow{\lambda_\hbar} M  \to  \bar{R}\otimes_R  M  \to 0$. Apply Lemma \ref{Rees-lemma-regular} to the 1st and 4th term of the bottom row, the result follows.
\end{proof}

\begin{lemma}\label{spectral-sequence}
Assume $R$ is left and right $G$-noetherian and $M\in \smod^GR$ is $\hbar$-torsionfree. Then there is a sequence of complexes $(E_1^*,d_1^*), (E_2^*,d_2^*),\cdots$ in $\smod^G\bar{R}^o$ such that for every $i\in \Z$ one has:
\begin{enumerate}
\item $E_1^i \cong \Sigma_{-i\varepsilon} \uExt_{\bar{R}}^{i,G}(\bar{R}\otimes_RM,\bar{R})^{{{\tau_\hbar}}^{i}} $ in $\Mod^G\bar{R}^o$;
\item $E_r^i \cong \Sigma_{-i\varepsilon}H^i(E_{r-1}^*,d_{r-1}^*)^{{{\tau_\hbar}}^{i}}$ in $\Mod^G\bar{R}^o$ for every $r\geq2$;
\item $E_r^i \cong \Sigma_{-ir\varepsilon} F_\hbar(\uExt_R^{i,G}(M,R))\otimes_R\bar{R}^{{{\tau_\hbar}}^{ir}}$ in $\Mod^G\bar{R}^o$ for all $r\gg1$.
\end{enumerate}
\end{lemma}

The demonstration of this lemma involves the spectral sequence technique.

\begin{proof}
Choose a projective resolution $P^*\to M$ of $M$ in $\Mod^GR$ with each term finitely generated. Then $\bar{R}\otimes_RP^*\to \bar{R}\otimes_R M$ is a projective resolution of $\bar{R}\otimes_R M$ in $\Mod^G\bar{R}$ and $\uHom_R^G(P^*,R)$ is a complex in $\Mod^GR^o$ with each  term  finitely generated and $\hbar$-torsionfree. Note that $$\uHom_R^G(P^*,R)\otimes_R\bar{R} \cong \uHom_{\bar{R}}^G(\bar{R}\otimes_RP^*,\bar{R})$$
as complexes in $\Mod^G\bar{R}^o$. To save notations we write $(Q^*,d^*):= \uHom_{R}^G(P^*,R)$. For subgroups $N\subseteq Q^m$  and integers $n\geq 1$, we set $N\hbar^{-n}:= \{\,x\in Q^m\,|\, x\hbar^n\in N\,\}$. For all $r\geq1$ and all $p, q\in \Z$, we define
$$
E_r^{p,q} := \frac{Q^{p+q}\hbar^p \cap (d^{p+q})^{-1}(Q^{p+q+1}\hbar^{p+r} ) + Q^{p+q}\hbar^{p+1}}{Q^{p+q}\hbar^p \cap \im (d^{p+q-1})\, \hbar^{p+1-r}+ Q^{p+q}\hbar^{p+1}}.
$$
Then the differential $d^*$ induces morphisms  $d_r^{p,q}: E_r^{p.q} \to E_r^{p+r,q-r+1}$ in $\Mod^G\bar{R}^o$. It is tedious but straightforward to check that  $d_r^{p,q}\circ d_r^{p-r,q+r-1}=0$ and $\ker (d_r^{p,q})/\im(d_r^{p-r,q+r-1}) \cong E_{r+1}^{p,q}$.

For all $r\geq1$ and all $n\in \Z$, let $E_r^n= E_r^{nr,n(1-r)}$ and $d_r^n= d_r^{nr,n(1-r)}$. We are going to show the sequence of complexes $(E_1^*,d_1^*), (E_2^*,d_2^*),\cdots$ in $\smod^G\bar{R}^o$ fulfills the requirements. Clearly,
$$
E_1^i \cong \Sigma_{-i\varepsilon} (E_1^{0,i})^{{{\tau_\hbar}}^i} \cong \Sigma_{-i\varepsilon} H^i(Q^*\otimes_R\bar{R})^{{{\tau_\hbar}}^i} \cong \Sigma_{-i\varepsilon} \uExt_{\bar{R}}^{i,G} (\bar{R}\otimes_RM,\bar{R})^{{{\tau_\hbar}}^i}
$$
and
$$
E_r^i \cong \Sigma_{-i\varepsilon} (E_r^{i(r-1),i(2-r)})^{{{\tau_\hbar}}^i} \cong \Sigma_{-i\varepsilon}H^i(E_{r-1}^*,d_{r-1}^*)^{{{\tau_\hbar}}^i}
$$
for all $r\geq 2$ in $\Mod^G\bar{R}^o$, so it remains to show the requirement (3) holds.

It is easy to check that the $\bar{G}$-graded ring $A:=\sum\nolimits_{n\geq0} (R\hbar^n)t^n =R[\hbar t] \subseteq R[t]$, where $\bar{G}=\Z\times G$, is left and right $\bar{G}$-noetherian. So the $\bar{G}$-graded right $A$-module $U^{i+1} := \sum_{n\geq 0} (Q^{i+1}\hbar^n) t^n=Q^{i+1}[\hbar t] \subseteq Q^{i+1}[t]$ and the submodule $V^{i+1}: = \sum_{n\geq 0} (\im (d^i)\cap Q^{i+1}\hbar^n)t^n \subseteq U^{i+1}$ are $\bar{G}$-noetherian. Choose $\bar{G}$-homogeneous elements $x_1t^{m_1} ,\cdots,x_st^{m_s} $  such that $V^{i+1}=\sum_{j=1}^sx_jt^{m_j}A$. Then for all $r> \max\{\,m_1,\cdots,m_s\,\}$, one has
$\im (d^i)\cap Q^{i+1}\hbar^r = \sum\nolimits_{j=1}^s x_jR\hbar^{r-m_i} \subseteq \im(d^i)\, \hbar$ which yields that
$$(d^i)^{-1}( Q^{i+1}\hbar^r)+ Q^i\hbar = \ker(d^i)+ Q^i\hbar. $$
Indeed, if $d^i(y)\in Q^{i+1}\hbar^{r}$ then $d^i(y)=d^i(z)\hbar$ for some $z\in Q^i$ and hence $y=y-z\hbar +z\hbar\in \ker(d^i)+Q^i\hbar$. In addition, since $Q^i$ is $G$-noetherian, one has for $r\gg1$ that
$$
\im(d^{i-1})\hbar^{1-r} + Q^i\hbar = \cup_{n\geq1} ( \im(d^{i-1})\hbar^{1-n} + Q^i\hbar).
$$
Therefore, we conclude for  $r\gg1$ that
\begin{eqnarray*}
E_r^{0,i} = \frac{\ker(d^i)+Q^i\hbar }{\cup_{n\geq1}(\im(d^{i-1})\hbar^{1-n} + Q^i\hbar)} \cong \frac{\ker(d^i)}{\cup_{n\geq1}\im(d^{i-1})\hbar^{1-n} + \ker(d^i)\hbar} \cong F_\hbar(H^i(Q^*,d^*))\otimes_R\bar{R},
\end{eqnarray*}
where the first isomorphism is established from the observation $ \ker(d^i)\cdot\hbar = \ker(d^i)\cap (Q^i\cdot\hbar) $.  Finally, since $E_r^i\cong \Sigma_{-ir\varepsilon} (E_r^{0,i})^{{{\tau_\hbar}}^i}$ in $\Mod^G\bar{R}^o$,  the requirement (3) is satisfied.
\end{proof}

\begin{theorem}\label{Auslander-cut-at-infinity-general}
Suppose that $R$ is left and right $G$-noetherian. Then
\begin{enumerate}
\item $R$ is $G$-Gorenstein if and only if $\bar{R}$ and $\hat{R}$ are so.
\item $R$ is $G$-Auslander-Gorenstein if and only if $\bar{R}$ and $\hat{R}$ are so.
\item $R$ is $G$-Auslander-regular provided that $\bar{R}$ and $\hat{R}$ are so.
\end{enumerate}
\end{theorem}

This theorem stems from \cite[Chapter III, Section 3.1, Theorem 6]{LV}, which deals with the Auslander-regularity in the ungraded setting under the stronger condition that $\hbar$ is regular central.

\begin{proof}
Clearly, (2) is a direct consequence of (1) and Proposition \ref{idim-cut-at-infinity-general}, and  (3) is a direct consequence of (1) and Proposition \ref{gldim-cut-at-infinity-general}.
By symmetry, it remains to show that  $\smod^G_{au}R=\smod^GR$ if and only if $\smod^G_{au}\bar{R}=\smod^G\bar{R}$ and $\smod^G_{au}\hat{R}=\smod^G\hat{R}$.
This can be read from the following three claims:
\begin{description}
\item[Claim 1] For any module $J\in\smod^G\bar{R}$,  $J\in \smod^G_{au}\bar{R}$ if and only if  $J\in \smod^G_{au}R$.
\item[Claim 2] For any module $K\in \smod^G R$, if $K\in \smod^G_{au}R$ then  $\hat{R}\otimes_RK \in \smod^G_{au}\hat{R}$.
\item[Claim 3] For any $\hbar$-torsionfree module $M\in \smod^GR$, if $\bar{R}\otimes_RM\in \smod^G_{au}\bar{R}$ and $\hat{R}\otimes_RM\in \smod^G_{au} \hat{R}$ then $M\in \smod^G_{au}R$.
\end{description}
Indeed, the desired forward implication is clear by Claim 1 and Claim 2. To see the desired converse implication, we assume $\smod^G_{au}\bar{R}=\smod^G\bar{R}$ and $\smod^G_{au}\hat{R}=\smod^G\hat{R}$. Fix an arbitrary  module $E\in \smod^GR$. Claim 1 tells us that  $\hbar^nT_\hbar (E)/\hbar^{n+1}T_\hbar (E)\in \smod^G_{au}R$ for every $n\geq0$. Since $\hbar^rT_\hbar(E)=0$ for some $r\geq1$, it follows that $T_\hbar(E)\in \smod^G_{au}R$ by Lemma \ref{Auslander-extension}. Also, by Claim 3, one has $F_\hbar(E)\in \smod^G_{au}R$. Apply Lemma \ref{Auslander-extension} again, one gets $E\in \smod^G_{au}R$. Thus, $\smod^G_{au}R=\smod^GR$ as required.

Now we proceed to justify the above three claims.

\textbf{Proof of Claim 1.} This is an easy consequence of Lemma \ref{Rees-lemma-regular}. We leave it to readers.

\textbf{Proof of Claim 2.}  Let $U$ be an arbitrary homogeneous $\hat{R}^o$-submodule of
$\uExt_{\hat{R}}^{n,G}(\hat{R}\otimes_RK,\hat{R})\cong \uExt_R^{n,G}(K,R)\otimes_R\hat{R}$.  Then there is a homogeneous $R^o$-submodule $K'$ of $\uExt_R^{n,G}(K,R)$ such that $U\cong K'\otimes_R\hat{R}$.  Therefore
$\uExt_{\hat{R}^o}^{i,G}(U,\hat{R}) \cong \hat{R}\otimes_R \uExt_{R^o}^{i,G}(K',R) = 0$ for all $i<n$. Thus $\hat{R}\otimes_RK\in \smod^G_{au}\hat{R}$.

\textbf{Proof of Claim 3.} By Lemma \ref{grade-short-exact-sequence}, to see $M\in \smod^G_{au}R$  it suffices to show every homogeneous $R^o$-submodule of $T_n=T_\hbar(\uExt_R^{n,G}(M,R))$  and of $F_n=F_\hbar(\uExt_R^{n,G}(M,R))$ have grade $\geq n$.

Let $N$ be an arbitrary homogeneous $R^o$-submodule of $T_n$. Let $N_{i}=\{\, x\in N\,|\, x\hbar^i=0\,\}$ for all $i\geq0$. One has $0=N_{0}\subseteq N_{1}\subseteq \cdots \subseteq N_{r} = N$ for some $r\gg0$ because $N$ is $G$-noetherian. The right multiplication by $\hbar$ induces injective morphisms in $\Mod^GR^o$ as follows:
$$ \cdots \hookrightarrow \Sigma_{-2\varepsilon}(N_3/N_{2})^{{{\tau_\hbar}}^{2}} \hookrightarrow \Sigma_{-\varepsilon}(N_{2}/N_{1})^{{\tau_\hbar}}\hookrightarrow N_1/N_0.$$
Then, for $i\geq0$, $\Sigma_{-i\varepsilon}(N_{i+1}/N_i)^{{{\tau_\hbar}}^i}$ is a subquotient of $\Sigma_{\varepsilon}\uExt_{\bar{R}}^{n-1,G}(\bar{R}\otimes_R M,\bar{R})^{\tau_\hbar^{-1}}$ by Lemma \ref{long-exact-sequence-regular} and thereof $\grad_{R^o}^GN_{i+1}/N_i = \grad_{\bar{R}^o}^GN_{i+1}/N_i +1 \geq n$ by Lemma \ref{Rees-lemma-regular}. Thus $\grad_{R^o}^GN \geq n$ by Lemma \ref{grade-short-exact-sequence}.

Let $L$ be an arbitrary homogeneous $R^o$-submodule of $F_n$. Let
$L'=\{\, x\in F_n\,|\, x\hbar^i\in L \text{ for }i\gg0\}$. Apply the long $\uExt$-sequence associated to the exact sequence $0\to L\to L'\to L'/L\to 0$, to see $\grad_{R^o}^GL\geq n$ it suffices to show $\grad_{R^o}^GL'\geq n$ and $\grad_{R^o}^GL'/L\geq n+1$.

By Lemma \ref{spectral-sequence}, we have  that $F_n\otimes_R\bar{R}$ is a subquotient of $\uExt_{\bar{R}}^{n,G}(\bar{R}\otimes_RM,\bar{R})$ in $\Mod^G\bar{R}^o$. We also have $F_n\otimes_R\hat{R} \cong \uExt_R^{n,G}(M,R)\otimes_R\hat{R} \cong \uExt_{\hat{R}}^{n,G}(\hat{R}\otimes_RM,\hat{R})$ in $\Mod^G\hat{R}^o$.
Since $L'\otimes_R\bar{R}$ is a submodule of $F_n\otimes_R\bar{R}$ (because $F_n/L'= F_\hbar(F_n/L)$ is $\hbar$-trosionfree)
and $L'\otimes_R\hat{R}$ is a submodule of $F_n\otimes_R\hat{R}$, it follows that
$\grad_{R^o}^GL'\geq \min \{\, \grad_{\bar{R}^o}^GL'\otimes_R \bar{R}, \, \grad_{\hat{R}^o}^GL'\otimes_R\hat{R}\, \}\geq n.$

Let $L'_i= L'\hbar^i+L$ for all $i\geq0$. One has  $L'=L'_0\supseteq L'_1\supseteq \cdots \supseteq L'_r=L$ for some $r\gg0$ because $L'/L$ is $G$-noetherian.
The right multiplication by $\hbar$ induces surjective morphisms in $\Mod^GR^o$ as follows:
$$ L'_0/L'_1 \twoheadrightarrow \Sigma_\varepsilon(L'_1/L'_2)^{{{\tau_\hbar}}^{-1}} \twoheadrightarrow \Sigma_{2\varepsilon}(L'_1/L'_2)^{{{\tau_\hbar}}^{-2}} \twoheadrightarrow \cdots. $$
Since $L_0'/L_1'\cong L'/L\otimes_R\bar{R}$ in $\Mod^G\bar{R}^o$, $L_0'/L_1'$ is a subquotient of $F_n/L\otimes_R\bar{R}$. Consequently, for $i\geq0$, $\Sigma_{i\varepsilon}(L'_i/L'_{i+1})^{{{\tau_\hbar}}^{-i}}$ is a subquotient of $F_n\otimes_R\bar{R}$ and thereof
$\grad_{R^o}^G L'_i/L'_{i+1} = \grad_{\bar{R}^o}^GL'_i/L'_{i+1} +1  \geq n+1$ by Lemma \ref{Rees-lemma-regular}. Thus $\grad_{R^o}^GL'/L  \geq n+1$ by Lemma \ref{grade-short-exact-sequence}.
\end{proof}

\begin{proposition}\label{grade-cut-at-infinity}
Suppose $R$ is left and right $G$-noetherian, and $M\in \smod^GR$ is $\hbar$-torsionfree.  Then
\begin{enumerate}
\item $\grad_{\hat{R}}^G \hat{R}\otimes_RM \geq \grad_{R}^GM \geq \grad_{\bar{R}}^G\bar{R}\otimes_R M$ when $\hbar\in J^G(R)$.
\item $ \grad_{\hat{R}}^G \hat{R}\otimes_RM = \grad_{R}^GM \leq \grad_{\bar{R}}^G\bar{R}\otimes_R M$ when $\bar{R}$ is $G$-Auslander-Gorenstein.
\end{enumerate}
\end{proposition}

\begin{proof}
Let $p:=\grad^G_{\bar{R}}\bar{R}\otimes_R M$. Then, by Lemma \ref{long-exact-sequence-regular}, $\uExt_R^{i,G}(M,R)$ is $\hbar$-torsionfree for every integer $i\leq p$ and $\uExt_R^{i,G}(M,R)\cdot \hbar =\uExt_R^{i,G}(M,R)$ for every integer $i<p$.

(1) Assume $\hbar\in J^G(R)$.  By the graded version of the Nakayama lemma, $\grad_R^GM$ and $\grad_{\hat{R}}^G\hat{R}\otimes_RM$ are both $\geq p$. Hence, to see the result, we may assume $p<\infty$. Then, by Lemma \ref{long-exact-sequence-regular}, $\grad_R^GM$ has only two possible choices, which are $p$ and $p+1$. The desired result now follows directly in both cases.

(2) Assume $\bar{R}$ is $G$-Auslander-Gorenstein. Note that the groups $\uExt_R^{i,G}(M,R)$ and
$\uExt_{\hat{R}}^{i,G}(\hat{R}\otimes_RM,\hat{R})$ are both zero or both nonzero for every integer $i\leq p$. Thus, we may assume $p<\infty$.

By Lemma \ref{spectral-sequence}, there is a sequence of monomorphisms in $\Mod^G\bar{R}^o$ as follows
$$
\cdots \xrightarrow{f_3} \Sigma_{3p\varepsilon} (E_3^p)^{{{\tau_\hbar}}^{-3p}} \xrightarrow{f_2} \Sigma_{2p\varepsilon} (E_2^p)^{{{\tau_\hbar}}^{-2p}}\xrightarrow{f_1} \Sigma_{p\varepsilon} (E_1^p)^{{{\tau_\hbar}}^{-p}} \cong \uExt_{\bar{R}}^{p,G}(\bar{R}\otimes_RM, \bar{R})
$$
such that
\begin{itemize}
\item $\coker(f_r)$ is a subquotient of $\Sigma_{-r\varepsilon} \uExt_{\bar{R}}^{p+1,G}(\bar{R}\otimes_RM,\bar{R})^{{{\tau_\hbar}}^r}$  for all $r\geq 1$; and
\item $\Sigma_{rp\varepsilon} (E_r^p)^{{{\tau_\hbar}}^{-rp}} \cong \uExt_R^{p,G}(M,R)\otimes_R\bar{R} $ for all $r\gg1$.
\end{itemize}
Indeed, use notations in Lemma \ref{spectral-sequence}, the monomorphism $f_r$ is just the composition
$$
E_{r+1}^p \xrightarrow{\cong} \Sigma_{-p\varepsilon} H^p(E_{r}^*, d_{r}^*)^{\tau_\hbar^p} = \Sigma_{-p\varepsilon} (\ker d_{r}^p)^{\tau_\hbar^p} \hookrightarrow  \Sigma_{-p\varepsilon} (E_{r}^p)^{\tau_\hbar^p}.
$$
Here, the second ``='' comes from the fact $E_{r}^{p-1}=0$.

Now, by the graded version of \cite[Proposition 1.6 (2), Proposition 1.8]{Bj}, we have
$$ p=\grad^G_{\bar{R}^o}E_1^p=\grad^G_{\bar{R}^o} E_{2}^p =\cdots = \grad_{\bar{R}^o}^G\uExt_R^{p,G}(M,R)\otimes_R\bar{R}.$$
Consequently, $\uExt_R^{p,G}(M,R)\neq 0$ and thereof $\grad_{\hat{R}}^G \hat{R}\otimes_RM = \grad_{R}^GM \leq p$.
\end{proof}

Next we focus on a special situation. We say  that the $G$-graded ring $R$ is $\hbar$-discrete if ${}_RR$ is $\hbar$-discrete, or equivalently $R_R$ is $\hbar$-discrete. It is easy to check that if $R$ is $\hbar$-discrete then $\hbar\in J^G(R)$.

\begin{proposition}(\cite[Lemma 1.11]{Ro})\label{noetherian-cut-at-infinity}
Suppose that $R$ is $\hbar$-discrete. Then $R$ is left (resp. right) $G$-noetherian if and only if $\bar{R}$ is too; and in this case so is $\hat{R}$.
\end{proposition}

\begin{proof}
We only need to show $R$ is left $G$-noetherian under the assumption that $\bar{R}$ is so. Suppose that $R$ has an infinitely generated homogeneous left ideal, then by Zorn's Lemma, we may choose a maximal one, say $L$. Note that $R/L$ is $G$-noetherian. Consider the exact sequence
$$
0\to \hbar L'/ \hbar L \to L/\hbar L \to L/\hbar L'\to 0,
$$
where $L'=\{x\in R\ |\ \hbar x\in L\}$. Clearly, $\hbar L'/\hbar L$ and $L/\hbar L' \cong (L+\hbar R)/\hbar R$ are finitely generated. Therefore  $L/\hbar L$ is finitely generated and hence $L$ is too (because $L$ is $\hbar$-discrete), a contradiction.
\end{proof}

\begin{theorem}\label{Auslander-cut-at-infinity-special}
Suppose that $R$ is $\hbar$-discrete. Then
\begin{enumerate}
\item $R$ is $G$-Gorenstein if and only if $\bar{R}$ is too; and in this case so is $\hat{R}$.
\item $R$ is $G$-Auslander-Gorenstein if and only if $\bar{R}$ is too; and in this case so is $\hat{R}$.
\item $\hat{R}$ and $R$ are $G$-Auslander-regular when $\bar{R}$ is so.
\end{enumerate}
\end{theorem}

\begin{proof}
By Proposition \ref{noetherian-cut-at-infinity}, one may assume in priori that $R$, $\bar{R}$ and $\hat{R}$ are all left and right $G$-noetheiran. Then, clearly, (2) is a direct consequence of (1) and Proposition \ref{idim-cut-at-infinity-special}, and (3) is a direct consequence of (1) and Proposition \ref{gldim-cut-at-infinity-special}. Now, by  Theorem \ref{Auslander-cut-at-infinity-general} (1), we are left to see ``$\bar{R}$ is $G$-Gorenstein implies $\hat{R}$ is too''. By symmetry, it suffices to show that for any $n\geq0$ and any $\hbar$-torsionfree module $M\in \smod^GR$, it follows that every homogenous $\hat{R}^o$-submodule $V$ of $\uExt_{\hat{R}}^{n,G}(\hat{R}\otimes_RM,\hat{R})$ has grade $\geq n$.

Let $D_n=\uExt_{\hat{R}}^{n,G}(\hat{R}\otimes_RM,\hat{R})$ and $F_n=F_\hbar(\uExt_R^{n,G}(M,R))$. One may identify $F_n$ with a homogeneous $R^o$-submodule of $D_n$ because $F_n\otimes_R\hat{R} \cong D_n$ in $\Mod^G\hat{R}^o$. Since $F_n/(V\cap F_n)$ is $\hbar$-torsionfree, $(V\cap F_n)\otimes_R\bar{R}$ is a submodule of $F_n\otimes_R\bar{R}$ and hence a subquotient of $\uExt_{\bar{R}}^n(\bar{R}\otimes_R M,\bar{R})$ by Lemma \ref{spectral-sequence}. Therefore, $$\grad^G_{\hat{R}^o}V = \grad_{\hat{R}^o}^G(V\cap F_n)\otimes_R\hat{R}\geq \grad^G_{\bar{R}^o}(V\cap F_n)\otimes_R\bar{R}\geq n,$$
where the first ``$\geq$'' used  Proposition \ref{grade-cut-at-infinity} (1) and the second ``$\geq$'' used Lemma \ref{grade-short-exact-sequence}.
\end{proof}

\section{Proof of the main results}

In this section, we  prove Theorems \ref{Auslander-regrading} and \ref{homo-dim-regrading-general}.

We employ the the following notations. For a group homomorphism $\varphi:G\to H$, let $r_\varphi$ denote the rank of $\ker\varphi$; and for an element $\varepsilon \in G$, let $\pi_\varepsilon:\bar{G}\to G$ be the group homomorphism given by $(n,\gamma)\mapsto n\varepsilon+\gamma$, where $\bar{G}=\Z\times G$. We also recall the following two conventions. For a $G$-graded ring $A$, the polynomial ring $A[t]$ and Laurant ring $A[t,t^{-1}]$ are $\bar{G}$-graded by $\deg(at^i) = (i,\deg(a))$ when $a\in A$ is homogeneous; and for a $G$-graded ring $A$ and a semigroup $\Omega$ of $G$, the semigroup ring $A[\Omega]$ is $G$-graded by $\deg(ae_\gamma) = \deg (a) +\gamma$ for homogeneous  elements $a\in A$ and indexes $\gamma\in \Omega$.

\begin{lemma}\label{group-ring-isomorphism}
Let $A$ be a $G$-graded ring and $\varepsilon\in G\backslash T(G)$. Then there are natural isomorphisms  $\pi_\varepsilon^*(A[t]) \cong A[\N\varepsilon]$ and $\pi_\varepsilon^*(A[t,t^{-1}]) \cong A[\Z\varepsilon] \cong A[\N\varepsilon][e_\varepsilon^{-1}]$ as $G$-graded rings. \hspace{\fill} $\Box$
\end{lemma}

The following lemma is crucial for our purpose.

\begin{lemma}\label{behavior-polynomial-extension}
Let $A$ be a $G$-graded ring and $\varepsilon\in G$.
\begin{enumerate}
\item If $A$ is left (resp. right) $G$-noetherian then so are $\pi_\varepsilon^*(A[t])$ and $\pi_\varepsilon^*(A[t,t^{-1}])$.
\item Assume $A$ is left $G$-noetherian. Then for every module $N\in\Mod^GA$, one has
$$
\idim_{\pi_\varepsilon^*(A[t,t^{-1}])}^G \pi_\varepsilon^*(A[t,t^{-1}])\otimes_AN \leq \idim_{\pi_\varepsilon^*(A[t])}^G \pi_\varepsilon^*(A[t]) \otimes_AN = \idim_A^G N +1.
$$
\item $\gldim(\Mod^G\pi_\varepsilon^*(A[t,t^{-1}])) \leq \gldim(\Mod^G\pi_\varepsilon^*(A[t])) = \gldim(\Mod^GA) +1$.
\item If $A$ is $G$-Gorenstein then so are $\pi_\varepsilon^*(A[t])$ and $\pi_\varepsilon^*(A[t,t^{-1}])$.
\item If $A$ is $G$-Auslander-Gorenstein  then so are $\pi_\varepsilon^*(A[t])$ and $\pi_\varepsilon^*(A[t,t^{-1}])$.
\item If $A$ is $G$-Auslander-regular then so are $\pi_\varepsilon^*(A[t])$ and $\pi_\varepsilon^*(A[t,t^{-1}])$.
\end{enumerate}
\end{lemma}

\begin{proof}
First note that $\pi_\varepsilon^*(A[t,t^{-1}]) \cong \pi_\varepsilon^*(A[t])[t^{-1}]$ as $G$-graded rings, and $\pi_\varepsilon$ is surjective with $\ker \pi_\varepsilon= \Z \eta$, where $\eta=(1,-\varepsilon)$. It is easy to check that $A[t]$ is $t$-discrete and $A[t][\N\eta]$ is $e_\eta$-discrete. Also, $A[t]/(t) \cong \iota^*(A)$ as $\bar{G}$-graded rings, where $\iota:G\to \bar{G}$ is the map given by $\gamma\mapsto (0,\gamma)$.

(1) This is the graded version of the Hilbert basis theorem. We give here an alternative demonstration. Suppose $A$ is left $G$-noetherian. By Lemma \ref{ext-regrading-injective}, $\iota^*(A)$ is left $\bar{G}$-noetherian. Apply Proposition \ref{noetherian-cut-at-infinity} to the situation $(R,\hbar)= (A[t], t)$, one obtains that $A[t]$ is left $G$-noetherian; then apply Proposition \ref{noetherian-cut-at-infinity}  to the situation $(R,\hbar)=(A[t][\N\eta], e_\eta)$, it follows that $A[t][\N\eta]$ and $A[t][\Z\eta]$ are left $\bar{G}$-noetherian. Now by Lemma \ref{ext-regrading-surjective}, $\pi_\varepsilon^*(A[t])$ is left $G$-noetherian and thereof so is $\pi_\varepsilon^*(A[t,t^{-1}])$.

(2) By above discussion, $A[t]$ and $A[t][\N\eta]$ are left $\bar{G}$-noetherian. So we have
\begin{equation*}
\begin{split}
\idim_{\pi_\varepsilon^*(A[t])}^{G} \pi_\varepsilon^*(A[t]) \otimes_A N
& =  \idim_{A[t][\Z\eta]}^{\bar{G}} A[t][\Z\eta]\otimes_{\iota^*(A)}\iota_A^*(N) \\
& \leq  \idim_{A[t]}^{\bar{G}} A[t] \otimes_{\iota^*(A)}\iota_A^*(N)\\
& =  \idim_{\iota^*(A)}^{\bar{G}} \iota_A^*(N) +1 \\
&=   \idim_A^GN + 1.
\end{split}
\end{equation*}
Here, the first ``='' used Lemma \ref{ext-regrading-surjective}, the ``$\leq$'' and the second ``='' used Proposition \ref{idim-cut-at-infinity-special}, and the final ``='' used Proposition \ref{homo-dim-regrading-basic} (3). Also, $\pi_\varepsilon^*(A[t])$ is left $G$-noetherian by (1). Thereof, we have $$\idim_{\pi_\varepsilon^*(A[t])}^G \pi_\varepsilon^*(A[t]) \otimes_AN = \max\{\, \idim_A^G N +1, \, \idim_{\pi_\varepsilon^*(A[t,t^{-1}])}^G \pi_\varepsilon^*(A[t,t^{-1}])\otimes_AN\, \}$$
by applying Proposition \ref{idim-cut-at-infinity-general} to the situation $(R,\hbar) = (\pi_\varepsilon^*(A[t]), t)$. The result now follows.

(3) The ``$\leq$'' is clear and the ``='' is the graded version of  \cite[Theorem 7.5.3]{MR}.

(4) Assume $A$ is $G$-Gorenstein. By  Lemma \ref{ext-regrading-injective}, $\iota^*(A)$ is $\bar{G}$-Gorenstein. Apply Theorem \ref{Auslander-cut-at-infinity-special} (1) to the situation $(R,\hbar)= (A[t], t)$ and then to the situation $(R,\hbar)=(A[t][\N\eta], e_\eta)$, it follows that $A[t][\Z\eta]$ is $\bar{G}$-Gorenstein. Now by Lemma \ref{ext-regrading-surjective}, $\pi_\varepsilon^*(A[t])$ is $G$-Gorenstein. Thereof, $\pi_\varepsilon^*(A[t,t^{-1}])$ is also $G$-Gorenstein by applying Theorem \ref{Auslander-cut-at-infinity-general} (1) to the situation  $(R,\hbar) = (\pi_\varepsilon^*(A[t]), t)$.

Finally, (5) follows directly from (2) and (4), and (6) follows directly from (3) and (4).
\end{proof}

\begin{remark}
In the work \cite{Ek}, by using the technique of filtered rings and the external homogenization,  similar results of Lemma \ref{behavior-polynomial-extension} were established in the special situation that $G=\Z$ and $\varepsilon=1$. However, these two techniques loose their effectiveness in the general case.
\end{remark}

Now we are ready to prove the main results. First we deal with Theorem \ref{homo-dim-regrading-general}.

\begin{proof}[Proof of Theorem \ref{homo-dim-regrading-general}] \hspace{\fill}

(1) The first ``$\leq$''  is by Proposition \ref{homo-dim-regrading-basic} (2). We proceed to see the second ``$\leq$''  by induction on $r_\varphi$.

First consider the case  $r_\varphi=0$. It is clear by Proposition \ref{homo-dim-regrading-basic} (3).

Now suppose $r_\varphi>0$. Factor the map $\varphi$ into $G\xrightarrow{\psi} G' \xrightarrow{\varphi'} H$ such that $\psi$ is surjective, $\ker\psi = \Z\varepsilon$ for some $\varepsilon \in \ker\varphi\backslash T(\ker\varphi)$ and $T(\ker\varphi') \cong T(\ker\varphi)$. Then for every module $N\in \Mod^GA$, we have
\begin{equation*}
\idim_{\psi^*(A)}^{G'}\psi_{A}^*(N)
=  \idim_{A[\Z\varepsilon]}^G A[\Z\varepsilon]\otimes_AN = \idim_{\pi_\varepsilon^*(A[t,t^{-1}])}^G \pi_\varepsilon^*(A[t,t^{-1}])\otimes_AN \leq   \idim_A^GN + 1.
\end{equation*}
Here, the first ``$=$'' used Lemma \ref{ext-regrading-surjective}, the second ``='' used Lemma \ref{group-ring-isomorphism} and the ``$\leq$'' used  Lemma \ref{behavior-polynomial-extension} (2). Also, by Lemma \ref{group-ring-isomorphism} and Lemma \ref{behavior-polynomial-extension} (1), $A[\Z\varepsilon]$ is left $G$-noetherian; hence, by Lemma \ref{ext-regrading-surjective}, $\psi^*(A)$ is left $G'$-notherian.   Finally, since $r_{\varphi'} = r_\varphi-1$, the desired ``$\leq$'' follows by the induction hypothesis.

(2) The first ``$\leq$''  is by Proposition \ref{homo-dim-regrading-basic} (4). We proceed to see the second ``$\leq$''  by induction on $r_\varphi$.

First consider the case $r_\varphi=0$. We may assume  $\gldim(\Mod^GA) = d <\infty$. In addition, by Lemma \ref{ext-regrading-injective}, we may assume further that $\varphi$ is surjective. Let $M\in\Mod^GA[\ker\varphi]$ and let
$$
\cdots \to P_n \xrightarrow{\partial_n} \cdots \xrightarrow{\partial_2} P_1 \xrightarrow{\partial_1} P_0 \xrightarrow{\partial_0} M\to 0
$$
be a projective resolution in $\Mod^GA[\ker\varphi]$. Clearly, it is also a projective resolution of $M$ in $\Mod^GA$. So $K= \im \partial_d$ is projective in $\Mod^GA$. Therefore, there exists a morphism $f\in \Hom_{\Mod^GA}(K,P_d)$ such that $\partial_d'\circ f = \id_K$, where $\partial_d':P_d\to K$  is the co-restriction of $\partial_d$.  Let $$\tilde{f}:= (|T(\ker \varphi)|\cdot 1_A)^{-1}\cdot \sum\nolimits_{\gamma\in \ker\varphi} \lambda_{e_\gamma}\circ f \circ \lambda_{e_{-\gamma}}.$$
It is easy to check that $\tilde{f} \in \Hom_{\Mod^GA[\ker\varphi]}(K,P_d)$ and $\partial_d'\circ \tilde{f} = \id_K$, so $K$ is projective in $\Mod^GA[\ker\varphi]$.
Hence from Lemma \ref{ext-regrading-surjective} we obtain $\gldim(\Mod^H\varphi^*(A)) = \gldim(\Mod^GA[\ker\varphi])\leq d$.

Now suppose $r_\varphi>0$. Factor the map $\varphi$ into $G\xrightarrow{\psi} G' \xrightarrow{\varphi'} H$ such that $\psi$ is surjective, $\ker\psi = \Z\varepsilon$ for some $\varepsilon \in \ker\varphi\backslash T(\ker\varphi)$ and $T(\ker\varphi') \cong T(\ker\varphi)$. Then we have
$$
\gldim(\Mod^{G'}\psi^*(A)) = \gldim(\Mod^GA[\Z\varepsilon]) =\gldim(\Mod^G\pi_\varepsilon^*(A[t,t^{-1}]))  \leq  \gldim(\Mod^GA) +1.
$$
Here, the first ``$=$'' used  Lemma \ref{ext-regrading-surjective}, the second ``='' used Lemma \ref{group-ring-isomorphism} and the ``$\leq$'' used Lemma \ref{behavior-polynomial-extension} (3). Finally, since $r_{\varphi'} = r_\varphi-1$,  the desired ``$\leq$'' follows by the induction hypothesis.
\end{proof}

Since Theorem \ref{homo-dim-regrading-general} is already proved, we are able to prove Theorem \ref{Auslander-regrading}.

\begin{proof}[Proof of Theorem \ref{Auslander-regrading}] \hspace{\fill}

Clearly, (2) is a direct consequence of (1) and Theorem \ref{homo-dim-regrading-general} (1), and (3) is a direct consequence of (1) and Theorem \ref{homo-dim-regrading-general} (2). The  converse implication of (1) is also clear by Lemma \ref{ext-regrading-basic}. We assume $A$ is $G$-Gorenstein and proceed to show the forward implication of (1) by induction on $r_\varphi$.

First consider the case $r_\varphi=0$. By Lemma \ref{ext-regrading-injective}, we may assume $\varphi$ is surjective. To save the notations, we let $\Omega=\ker\varphi$. It is easy to check that the map $\uHom_{A}^G(A[\Omega], A) \to A[\Omega]$ given by $$f\mapsto \sum_{\gamma\in \Omega} f(e_{-\gamma}) e_\gamma$$ is an isomorphism of $G$-graded $A[\Omega]$-bimodules. Therefore, for every integer $i\in \Z$ and every module $M\in \Mod^G{A[\Omega]}$, we have the following natural isomorphisms in  $\Mod^GA[\Omega]^o$:
 $$\uExt^{i,G}_{A[\Omega]}(M,A[\Omega]) \cong \uExt^{i,G}_{A[\Omega]}(M,\uHom_A^G(A[\Omega],A)) \cong \uExt^{i,G}_A(A[\Omega] \otimes_{A[\Omega]} M,A) \cong \uExt^{i,G}_A(M,A).$$  It follows readily that $A[\Omega]$ is $G$-Gorenstein. Thereof, by Lemma \ref{ext-regrading-surjective}, $\varphi^*(A)$ is $H$-Gorenstein.

Now suppose $r_\varphi>0$. Factor the map $\varphi$ into $G\xrightarrow{\psi} G' \xrightarrow{\varphi'} H$ such that $\psi$ is surjective, $\ker\psi = \Z\varepsilon$ for some $\varepsilon \in \ker\varphi\backslash T(\ker\varphi)$ and $T(\ker\varphi) \cong T(\ker\varphi')$.  By Lemma \ref{group-ring-isomorphism} and Lemma \ref{behavior-polynomial-extension} (4),  $A[\Z\varepsilon]$ is $G$-Gorenstein. Then Lemma \ref{ext-regrading-surjective} tells us that $\psi^*(A)$ is $G'$-Gorenstein. Finally, since $r_{\varphi'} = r_\varphi-1$, it follows that $\varphi^*(A)=\varphi'^*(\psi^*(A))$ is $H$-Gorenstein by the induction hypothesis.
\end{proof}

By the same proof strategy as above, one may readily recover the following classical result.

\begin{proposition}\label{noetherian-regrading}
Let $A$ be a $G$-graded ring and $\varphi:G\to H$ a group homomorphism. Then, $A$ is left (resp. right) $G$-noetherian if and only if $\varphi^*(A)$ is left (resp. right) $H$-noetherian. \hspace{\fill} $\Box$
\end{proposition}

Theorem \ref{homo-dim-regrading-general} can be strengthened as follow by imposing a mild condition on the grading structure.

\begin{proposition}\label{homo-dim-regrading-special}
Let $A$ be a $G$-graded ring and $\varphi:G\to H$ a group homomorphism.
Suppose that $p(\supp A) \subseteq \N^r$ and $\ker p\cap \ker \varphi = T(\ker\varphi)$ for some group homomorphism $p:G\to \Z^r$, $r>0$.
\begin{enumerate}
\item Assume $A$ is left $G$-noetherian. Then for every module $N\in\smod^GA$, one has $$\idim_{\varphi^*(A)}^H \varphi^*_A(N) = \idim_A^GN.$$
\item Assume that the number of elements of $T(\ker\varphi)$ is invertible in $A$. Then $$\gldim(\Mod^H \varphi^*(A)) = \gldim(\Mod^GA).$$
\end{enumerate}
\end{proposition}

\begin{proof}
We need different proof strategy. First factor the map $\varphi_0=\varphi$ into
$$G \xrightarrow{\varphi_r} \Z^r\times H \xrightarrow{g_r} \cdots \xrightarrow{g_2} \Z \times H \xrightarrow{g_1} H,$$
where $\varphi_r: \gamma \mapsto (p(\gamma), \varphi(\gamma))$, $g_1:(n,\gamma) \mapsto \gamma$ and $g_i: (n_1,\cdots, n_i,\gamma) \mapsto (n_1+n_2, n_3,\cdots, n_{i},\gamma)$ for $i\geq2$. Clearly, $\ker\varphi_r=T(\ker\varphi)$, and $\ker g_i = \Z\varepsilon_i$ with $\varepsilon_1=(1,0)$ and $\varepsilon_i= (1,-1,0,\cdots,0)$ for $i\geq 2$. Also, for  $i=1,\cdots, r-1$, let $\varphi_i= g_{i+1} \circ \cdots \circ g_r \circ \varphi_r$.

(1) By Lemma \ref{group-ring-isomorphism}, Lemma \ref{behavior-polynomial-extension} (1) and Proposition \ref{noetherian-regrading}, $\varphi_i^*(A)[\N\varepsilon_i]$ is left $(\Z^i\times H)$-noetherian. It is easy to check that $ \varphi_i^*(A)[\N\varepsilon_i] \otimes_{\varphi_i^*(A)} \varphi_{i,A}^*(N)$  is $e_{\varepsilon_i}$-torsionfree and $e_{\varepsilon_i}$-discrete. So one has
$$
\idim_{\varphi_{i-1}^*(A)}^{\Z^{i-1}\times H} \varphi_{i-1,A}^*(N) = \idim_{\varphi_i^*(A)[\Z\varepsilon_i]}^{\Z^i\times H} \varphi_i^*(A)[\Z\varepsilon_i] \otimes_{\varphi_i^*(A)} \varphi_{i,A}^*(N)  \leq   \idim_{\varphi_i^*(A)}^{\Z^i\times H} \varphi_{i,A}^*(N).
$$
Here, ``$=$'' used Lemma \ref{ext-regrading-surjective} and ``$\leq$'' used Proposition \ref{idim-cut-at-infinity-special} for $(R,\hbar)= (\varphi_i^*(A)[\N\varepsilon_i], e_{\varepsilon_i} )$. Thus,
$$\idim_{\varphi^*(A)}^H \varphi_A^*(N) \leq \idim_{\varphi_r^*(A)}^{\Z^r\times H} \varphi_{r,A}^*(N) = \idim_A^GN \leq \idim_{\varphi^*(A)}^H \varphi_A^*(N),$$
where  ``$=$'' used Proposition \ref{homo-dim-regrading-basic} (3) and the second ``$\leq$'' used Proposition \ref{homo-dim-regrading-basic} (2).

(2) Let $q_i:\Z^i\times H \to \Z$ be the map given by $(n,\gamma) \mapsto n$ for $i=1$ and $(n_1,\cdots n_i, \gamma)\mapsto 2n_1+n_2+\cdots +n_i$ for $i\geq 2$. It is easy to check that $q_i(\supp \varphi_i^*(A)[\N\varepsilon_i]) \subseteq \N$ and $q_i(\varepsilon_i) =1$. So we have
$$
\gldim(\Mod^{\Z^{i-1}\times H}\varphi_{i-1}^*(A)) = \gldim(\Mod^{\Z^i\times H}\varphi_i^*(A)[\Z\varepsilon_i]) \leq \gldim(\Mod^{\Z^i\times H} \varphi_i^*(A)).
$$
Here,  ``$=$'' used Lemma \ref{ext-regrading-surjective} and ``$\leq$'' used Proposition \ref{gldim-cut-at-infinity-special}  for $(R,\hbar)= (\varphi_i^*(A)[\N\varepsilon_i], e_{\varepsilon_i} )$. Thus,
$$
\gldim(\Mod^H\varphi^*(A)) \leq \gldim(\Mod^{\Z^r\times H}\varphi_r^*(A)) = \gldim(\Mod^GA) \leq \gldim(\Mod^H\varphi^*(A)),
$$
where the ``$=$'' is by Theorem \ref{homo-dim-regrading-general} (2) and the second ``$\leq$'' is by  Proposition \ref{homo-dim-regrading-basic} (4).
\end{proof}

We finish this section with an expository  example on  injective and global dimensions.

\begin{example}
Let $A=\mathbb{C}[x,x^{-1}][y]$ be $\Z^2$-graded with $\deg(x)=(1,0)$ and $\deg(y) = (0,1)$. Let $\pi_i:\Z^2\to \Z$, $i=1,2$, be the coordinate projections. It is not hard to see the following equalities:
\begin{itemize}
\item $\idim_{A}^{\Z^2} A= \gldim(\Mod^{\Z^2}A) =1$,
\item $ \idim_{\pi_1^*(A)}^\Z \pi_1^*(A) = \gldim(\Mod^\Z\pi_1^*(A)) =1$,
\item $ \idim_{\pi_2^*(A)}^\Z \pi_2^*(A) = \gldim(\Mod^\Z\pi_2^*(A)) =2$.
\end{itemize}
We leave the proof to readers as flexible applications of the previous results.
\end{example}

\section{On the Cohen-Macaulay property}

Throughout this section all algebras are over a fixed field $\mathbb{K}$. Instead of the Auslander condition, the Cohen-Macaulay property on graded algebras are examined  from our viewpoint in this section.  We refer to \cite{KL,MR} for an exposition of the Gelfand-Kirillov dimension (GK-dimension, for short).

Given a $G$-graded algebra $A$ of finite GK-dimension, we denote by $\smod^G_{cm}A$  the full subcategory of $\smod^GA$  consisting of all objects $M$  such that $\gkdim_AM=\gkdim_A A - \grad_A^GM$.

\begin{definition}\label{Coehn-Macaulay-definition}
We say that a $G$-graded algebra $A$ is  {\em $G$-Cohen-Macaulay} if
it is left and right $G$-noetherian, has finite GK-dimension,  $\smod^G_{cm}A=\smod^GA$ and $\smod^G_{cm}A^o=\smod^GA^o$.
\end{definition}

We say that a $G$-graded algebra $A$ is \emph{well-supported} if there is a group homomorphism $p: G\to \Z$ such that  $p(\supp A) \subseteq \N$ and $\supp A\cap p^{-1}(n)$ is a finite set for every integer $n\in \Z$. Note that a $\Z^r$-graded algebra with support contained in $\N^r$ is obviously well-supported.

Recall that a $G$-graded vector space $V$ is called \emph{locally finite} if each piece $V_\gamma $ is of finite dimension. For a locally finite $\Z$-graded vector space $V$, we define a map  $d_V: \N\to \mathbb{R}$ by $n\mapsto \sum_{|i|\leq n} \dim V_i$.

\begin{lemma}\label{Cohen-Macaulary-extension}
Let $A$ be a well-supported and locally finite $G$-graded algebra. Suppose that $A$ is of finite GK-dimension and $G$-Auslander-Gorenstein. Then for any exact sequence $0\to L\to M\to N\to 0$ in $\smod^GA$, it follows that if $L$ and $N$ are in  $\smod^G_{cm}A$ then $M$ is too.
\end{lemma}

\begin{proof}
Let $p: G\to \Z$  be a group homomorphism such that $\supp A \subseteq p^{-1}(\N)$ and $\supp A \cap p^{-1}(n)$ is finite for all $n\in \Z$. Then, clearly, $p^*(A)$ is locally finite and finitely generated.  So we have
\begin{equation*}
\begin{split}
\gkdim_AM
&=\limsup\nolimits_{n\to\infty} \log_n d_{p_A^*(M)}(n) \\
&= \max \{\ \limsup\nolimits_{n\to\infty} \log_n d_{p_A^*(L)}(n),\, \limsup\nolimits_{n\to\infty} \log_n d_{p_A^*(N)}(n) \ \}\\
&=\max \{\ \gkdim_A L,\ \gkdim_A N \ \}.
\end{split}
\end{equation*}
Here, the first and the third ``='' used  \cite[Lemma 6.1 (b)]{KL}.  Also, we have $$\grad_A^GM= \inf \{\, \grad_A^GL,\, \grad_A^GN\, \}$$  by the graded version of \cite[Proposition 1.8]{Bj}. The result now follows.
\end{proof}

In the sequel, we use the notations and conventions in Section \ref{thesection-3}. In particular, $R$ is a $G$-graded algebra and $\hbar\in R$ is a homogeneous regular normal non-invertible element of degree $\varepsilon$.

\begin{lemma}(\cite[Lemma 5.7]{Lev})\label{gkdim-cut-at-infinity}
Assume $R$ is well-supported, locally finite and finitely generated, and assume $\varepsilon\notin T(G)$. Then for every $\hbar$-torsionfree module $M\in\smod^GR$, it follows that
$$
\gkdim_RM=\gkdim_{\bar{R}}\bar{R}\otimes_RM +1.
$$
\end{lemma}

\begin{proof}
Let $p: G\to \Z$  be a group homomorphism such that $\supp A \subseteq p^{-1}(\N)$ and $\supp A \cap p^{-1}(n)$ is finite for all $n\in \Z$. Then $p^*(A)$ is positively graded and locally finite, and $a:=p(\varepsilon) >0$. Replace $M$ by $\Sigma_{-r\varepsilon}M$ for some $r\gg0$, we may assume $\supp p_R^*(M)\subseteq \N$.  Then
\begin{eqnarray*}
\begin{split}
\gkdim_RM
& = \limsup\nolimits_{n\to\infty} \log_n d_{p^*_R(M)} (n)\\
& = \limsup\nolimits_{n\to \infty} \log_n \sum\nolimits_{i=0}^{[\frac{n}{a}]} d_{p^*_{\bar{R}}(\bar{R}\otimes_R M)} (n-ia) \\
& \leq \limsup\nolimits_{n\to \infty} \log_n d_{p^*_{\bar{R}}(\bar{R}\otimes_R M)} (n) + 1 \\
& = \gkdim_{\bar{R}} \bar{R}\otimes_RM +1.
\end{split}
\end{eqnarray*}
Here, $[\frac{n}{a}]$ denotes the integral part of $n/a$, and the first and the final ``='' used \cite[Lemma 6.1 (b)]{KL}. The desired equality now follows by \cite[Proposition 8.3.5]{MR}.
\end{proof}

\begin{theorem}\label{Cohen-Macaulary-cut-at-infinity-quotient}
Assume $R$ is well-supported and locally finite, and assume $\varepsilon\notin T(G)$. Then $R$ is $G$-Auslander-Gorenstein and $G$-Cohen-Macaulay if and only if $\bar{R}$ is too.
\end{theorem}

\begin{proof}
By Theorem \ref{Auslander-cut-at-infinity-special} (2), one may assume in priori that $R$ and $\bar{R}$ are both $G$-Auslander-Gorenstein. Note in particular that in this case $R$ is finitely generated.

Now assume $R$ is $G$-Cohen-Macaulay. Then for every module $L\in\smod^G\bar{R}$ we have
\begin{eqnarray*}
\gkdim_{\bar{R}}L = \gkdim_RL = \gkdim_RR-\grad_R^GL
 = \gkdim_{\bar{R}} \bar{R} - \grad_{\bar{R}}^GL.
\end{eqnarray*}
Here, the final ``='' is by Lemma \ref{gkdim-cut-at-infinity} and Lemma \ref{Rees-lemma-regular}. Thus $\bar{R}$ is $G$-Cohen-Macaulay.

Next, we assume instead $\bar{R}$ is  $G$-Cohen-Macaulay and proceed to show $R$ is so. For every $\hbar$-torsionfree module $N\in \smod^GR$, we have
\begin{eqnarray*}
\gkdim_RN =\gkdim_{\bar{R}} \bar{R}\otimes_RN + 1 = \gkdim_{\bar{R}} \bar{R} - \grad_{\bar{R}}^G\bar{R}\otimes_RN +1 = \gkdim_RR-\grad_R^GN.
\end{eqnarray*}
Here, the first ``='' used Lemma \ref{gkdim-cut-at-infinity} and the third ``='' used  Lemma \ref{gkdim-cut-at-infinity} and Proposition \ref{grade-cut-at-infinity}. Also, for every module $L\in\smod^GR$ with $\hbar \cdot L=0$, we have
\begin{eqnarray*}
\gkdim_RL =\gkdim_{\bar{R}}L = \gkdim_{\bar{R}} \bar{R} - \grad_{\bar{R}}^GL = \gkdim_RR-\grad_R^GL.
\end{eqnarray*}
Here, the final ``='' used Lemma \ref{gkdim-cut-at-infinity} and Lemma \ref{Rees-lemma-regular}. Since such $N$ and $L$ generate $\smod^GR$ by extension, it follows that $R$ is indeed $G$-Cohen-Macaulay by Lemma \ref{Cohen-Macaulary-extension}.
\end{proof}

\begin{theorem}\label{Cohen-Macaulary-cut-at-infinity-localization}
Assume $\{\, \tau_\hbar^n(x)\, \}$ spans a finite dimensional subspace for every homogeneous element $x\in R$. Assume further $\bar{R}$ is $G$-Auslander-Gorenstein. Then, $\hat{R}$ is $G$-Cohen-Macaulay when $R$ is so.
\end{theorem}

The first assumption in this theorem is fulfilled if $\hbar$ is central or $R$ is locally finite.

\begin{proof}
Assume $R$ is $G$-Cohen-Macaulay. It is easy to check that all $\hbar^i$ are local normal in the sense of \cite{ASZ2}. Then for every $\hbar$-torsionfree module $M\in \smod^GR$, we have
$$
\gkdim_{\hat{R}} \hat{R}\otimes_RM = \gkdim_{R}M = \gkdim_RR- \grad_R^GM = \gkdim_{\hat{R}}\hat{R} - \grad_{\hat{R}} \hat{R}\otimes_RM.
$$
Here, the first ``='' is by  \cite[Lemma 2.3]{ASZ2} and the final ``='' comes from \cite[Lemma 2.3]{ASZ2} and Proposition \ref{grade-cut-at-infinity} (2). Thus, $\hat{R}$ is $G$-Cohen-Macaulay.
\end{proof}

\begin{lemma}\label{gkdim-tensor}
Let $A$ and $B$ be algebras. Let $M\in \Mod A$ and $N\in \Mod B$. Assume there is a finite dimensional subspace $U$ of $A$ that containing $1_A$ and a finite dimensional subspace $X$ of $M$ such that $\gkdim_AM= \lim_{n\to \infty} \log_n\dim (U^n X)$. Then
\begin{flalign*}
&& \gkdim_{A\otimes B} M\otimes N = &\gkdim_AM +\gkdim_BN. &
\end{flalign*}
\end{lemma}

\begin{proof}
It follows from an obvious modification of  the discussion for \cite[Proposition 3.11]{KL}.
\end{proof}

\begin{theorem}\label{Cohen-Macaulay-regrading}
Let $A$ be a $G$-graded algebra and $\varphi:G\to H$ a group homomorphism.
Suppose that $A$ is locally finite and there is an injective group homomorphism $p:G\to \Z^r$, $r>0$, such that $p(\supp A)\subseteq \N^r$. Then $A$ is $G$-Auslander-Gorenstein and $G$-Cohen-Macaulay if and only if $\varphi^*(A)$ is $H$-Auslander-Gorenstein and $H$-Cohen-Macaulay.
\end{theorem}

\begin{proof}
Lemma \ref{ext-regrading-basic} tells us readily that ``$\varphi^*(A)$ is $H$-Cohen-Macalay $\Rightarrow$ $A$ is $G$-Cohen-Macaulay''. Then, by Theorem \ref{Auslander-regrading} (2), it remains to show $\varphi^*(A)$ is $H$-Cohen-Macalay under the assumption that $A$ is $G$-Auslander-Gorenstein and $G$-Cohen-Macaulay. To this end, first factor the map $\varphi_0=\varphi$ into
$$G \xrightarrow{\varphi_r} \Z^r\times H \xrightarrow{g_r} \cdots \xrightarrow{g_2} \Z \times H \xrightarrow{g_1} H,$$
where $\varphi_r: \gamma \mapsto (p(\gamma), \varphi(\gamma))$, $g_1:(n,\gamma) \mapsto \gamma$ and $g_i: (n_1,\cdots, n_i,\gamma) \mapsto (n_1+n_2, n_3,\cdots, n_{i},\gamma)$ for $i\geq2$. Clearly, $\varphi_r$ is injective, and $\ker g_i = \Z\varepsilon_i$ with $\varepsilon_1=(1,0)$ and $\varepsilon_i= (1,-1,0,\cdots,0)$ for $i\geq 2$. For  $i=1,\cdots, r-1$, let $\varphi_i= g_{i+1} \circ \cdots \circ g_r \circ \varphi_r$. Also, let $q_1:\Z\times H\to \Z$ be the map given by  $(n,\delta) \mapsto n$ and for $i\geq 2$, let $q_i:\Z^i\times H \to \Z$ be given by $(n_1,\cdots n_i, \delta)\mapsto 2n_1+n_2+\cdots +n_i$.

By Lemma \ref{ext-regrading-injective} and \cite[Proposition 5.1 (a)]{KL}, $\varphi_r^*(A)$ is $(\Z^r\times H)$-Cohen-Macaulay. So it suffices  to see
\begin{eqnarray*}
\begin{split}
\varphi_i^*(A) \text{ is $(\Z^i\times H)$-Cohen-Macaulay }
& \Longrightarrow \varphi_i^*(A)[\Z\varepsilon_i] \text{ is $(\Z^i\times H)$-Cohen-Macaulay } \\
& \Longrightarrow \varphi_{i-1}^*(A) \text{ is $(\Z^{i-1}\times H)$-Cohen-Macaulay }
\end{split}
\end{eqnarray*}
for $i=1,\cdots,r$. Note that $\varphi_i^*(A)[\N\varepsilon_i]$ is well-supported ($q_i$ fulfills the requirements) and locally finite and $\varphi_i^*(A)$ is $(\Z^i\times H)$-Auslander-Gorenstein by Theorem \ref{Auslander-regrading} (2). Then the first ``$\Longrightarrow$'' follows by applying Theorem \ref{Cohen-Macaulary-cut-at-infinity-quotient} and Theorem \ref{Cohen-Macaulary-cut-at-infinity-localization} to the situation $(R,\hbar) = (\varphi_i^*(A)[\N\varepsilon_i], e_{\varepsilon_i})$.

Now we are going to show the second ``$\Longrightarrow$''. Suppose $\varphi_i^*(A)[\Z\varepsilon_i]$ is $(\Z^i\times H)$-Cohen-Macaulay. Then for every module $M\in \smod^{\Z^{i}\times H} \varphi_i^*(A)[\Z\varepsilon_i]$, one has
\begin{eqnarray*}
\begin{split}
\gkdim_{\varphi_{i-1}^*(A)} \Xi_{\varphi_i^*(A)}^{g_i} (M)
& = \gkdim_{\varphi_{i-1}^*(A)[t,t^{-1}]} \Xi_{\varphi_i^*(A)}^{g_i} (M) [t,t^{-1}] -1 \\
& = \gkdim_{\varphi_i^*(A)[\Z\varepsilon_i]} M -1 \\
& = \gkdim_{\varphi_i^*(A)[\Z\varepsilon_i]} \varphi_i^*(A)[\Z\varepsilon_i] -\grad_{\varphi_i^*(A)[\Z\varepsilon_i]}^{\Z^i\times H} M - 1 \\
& = \gkdim_{\varphi_{i-1}^*(A)} \varphi_{i-1}^*(A) - \grad_{\varphi_i^*(A)[\Z\varepsilon_i]}^{\Z^i\times H} M,
\end{split}
\end{eqnarray*}
where the first and the final ``$=$'' used Lemma \ref{gkdim-tensor}, and the second ``='' will be justified in the next paragraph. Then Lemma \ref{ext-regrading-surjective} tells us  that $\varphi_{i-1}^*(A)$ is $(\Z^{i-1}\times H)$-Cohen-Macaulay.

First we introduce a (ungraded) $\varphi_i^*(A)[\Z\varepsilon_i]$-module structure on $\Xi_{\varphi_i^*(A)}^{g_i} (M)[t,t^{-1}]$ via restriction of scalars along the isomorphism $\varphi_i^*(A)[\Z\varepsilon_i] \xrightarrow{\cong} \varphi_{i-1}^*(A)[t,t^{-1}] $ of (ungraded) algebras given by
$$\varphi_i^*(A)[\Z\varepsilon_i]_\gamma \ni  ae_{n\varepsilon_i} \,\,  \mapsto \,\,  a\cdot t^{q_i(\gamma)}.$$
Then it is not hard to see that  the linear map $M \to  \Xi_{\varphi_i^*(A)}^{g_i} (M) [t,t^{-1}]$ given by
$$
M_\gamma \ni m \,\,  \mapsto \,\, \bar{m} \cdot t^{q_i(\gamma)}
$$
is an isomorphism of $\varphi_i^*(A)[\Z\varepsilon_i]$-modules, where
$\bar{m}$ is the canonical image of $m$ in $\Xi_{\varphi_i^*(A)}^{g_i} (M) \cong M/e_{\varepsilon_i}M$. This isomorphism of modules gives the second ``='' above, and thereof the proof is completed.
\end{proof}

We do not know an answer of the following natural question.

\begin{question}
For a $G$-graded algebra $A$ and a group homomorphism $\varphi:G\to H$ with $\ker\varphi$ finite, whether $A$ is $G$-Cohen-Macaulay is equivalent to that $\varphi^*(A)$ is $H$-Cohen-Macaulay?
\end{question}

\section{Homo-filtrations}

The theory of filtered rings and filtered modules has been well developed in literatures. In this section, we generalize some classical results in this respect from our perspective.

By a {\em homo-filtration} on a $G$-graded ring  $A$ we mean an increasing sequence $F=\{\,F_nA\,\}_{n\in \Z}$ of homogeneous subgroups of $A$ such that  $1\in F_0A$, $\cup_{n\in Z}F_nA=A$ and $F_mA\cdot F_nA\subseteq F_{m+n}A$ for all $m,\,n\in \Z$. The homo-filtration $F$ is called \emph{positive} if $F_nA=0$ for $n<0$, and $F$ is called \emph{locally discrete} if for every $\gamma \in G$ there is an integer $n=n(\gamma)$ such that $F_nA_\gamma = 0$. The $\bar{G}$-graded rings
$$
R_F(A): =\bigoplus\nolimits_{n\in\Z} F_nA\cdot t^n \ \subseteq \ A[t,t^{-1}]
$$
and
$$
G_F(A):=R_F(A)/(t) \cong
\bigoplus\nolimits_{n\in \Z} F_nA/F_{n-1}A.$$
are called respectively the \emph{Rees ring of $A$} and the \emph{associated graded ring of $A$} with respect to $F$. Note that $t\in R_F(A)$ is regular, central and homogeneous of degree $(1,0)$.

The next two trivial but key lemmas make it possible to lift information from  the associated graded ring to the homo-filtered ring in a unified and elegant way.

\begin{lemma}\label{Rees-ring-localization}
Let $A$ be a $G$-graded ring and $F$ a homo-filtration on $A$. Then,  as $\bar{G}$-graded rings,
\begin{flalign*}
&& R_F(A)[t^{-1}] \cong & A[t,t^{-1}]. & \Box
\end{flalign*}
\end{lemma}

\begin{lemma}\label{Laurent-equivalence}
Let $A$ be a $G$-graded ring. Then the functor $\Mod^GA\to \Mod^{\bar{G}}A[t,t^{-1}]$ given by $M\mapsto M[t,t^{-1}]$ is an equivalence of abelian categories. Moreover, there is a natural isomorphism $$\uExt_{A[t,t^{-1}]}^{i,\bar{G}}(M[t,t^{-1}],A[t,t^{-1}]) \cong \uExt_A^{i,G}(M,A)[t,t^{-1}]$$ in $\Mod^{\bar{G}}A[t,t^{-1}]^o$ for every integer $i\in \Z$ and every module  $M\in \Mod^GA$. \hspace{\fill} $\Box$
\end{lemma}

Now we proceed to deal with those properties that we concern.

\begin{proposition}
Let $A$ be a $G$-graded ring and $\varphi:G\to H$ a group homomorphism. Let $F$ be a locally discrete homo-filtration on $A$. Suppose that $G_F(A)$ is left $\bar{G}$-noetherian. Then $$\idim_{\varphi^*(A)}^H\varphi^*(A) \leq \idim_{G_F(A)}^{\bar{G}} G_F(A)\  + \text{ the rank of } \ker\varphi.$$
\end{proposition}

\begin{proof}
By Proposition \ref{noetherian-cut-at-infinity} and Lemma \ref{Rees-ring-localization}, $R_F(A)$ and $A[t,t^{-1}]$ are left $\bar{G}$-noetherian. It follows that $A$ is left $G$-noetherian by Lemma \ref{Laurent-equivalence} and
$$
\idim_A^GA =\idim_{A[t,t^{-1}]}^{\bar{G}} A[t,t^{-1}] \leq \idim_{G_F(A)}^{\bar{G}} G_F(A),
$$
where the ``='' is also by Lemma \ref{Laurent-equivalence} and the ``$\leq$'' comes from Lemma \ref{Rees-ring-localization} and Proposition \ref{idim-cut-at-infinity-special} (with $(R,\hbar)=(R_F(A),t)$). The result now follows immediately by Theorem \ref{homo-dim-regrading-general} (1).
\end{proof}

\begin{remark}
The injective dimension of a homo-filtered $G$-graded module (the definition is clear) over a homo-filtered $G$-graded ring can be characterized similarly.
\end{remark}

\begin{proposition}
Let $A$ be a $G$-graded ring and $\varphi:G\to H$ a group homomorphism. Let $F$ be a positive homo-filtration on $A$. Suppose  the number of elements of $T(\ker\varphi)$ is invertible in $A$. Then
$$
\gldim(\Mod^H\varphi^*(A)) \leq \gldim(\Mod^{\bar{G}} G_F(A))\ + \text{ the rank of } \ker\varphi.
$$
\end{proposition}

\begin{proof}
Let $p:\bar{G}\to \Z$ be the map given by $(n,\gamma)\mapsto n$. Clearly,
$p(\supp R_F(A))  \subseteq \N$ and $p(1,0)>0$. It follows that
$$
\gldim(\Mod^GA) = \gldim(\Mod^{\bar{G}}A[t,t^{-1}]) \leq \gldim(\Mod^{\bar{G}} G_F(A)),
$$
where the ``='' used Lemma \ref{Laurent-equivalence} and the ``$\leq$'' comes from Lemma \ref{Rees-ring-localization} and Proposition \ref{gldim-cut-at-infinity-special} (with $(R,\hbar) = (R_F(A),t)$). The result now follows immediately by Theorem \ref{homo-dim-regrading-general} (2).
\end{proof}

\begin{proposition}\label{Auslander-filtration}
Let $A$ be a $G$-graded ring and $\varphi:G\to H$ a group homomorphism. Let $F$ be a locally discrete homo-filtration on $A$. Then, if $G_F(A)$ is $\bar{G}$-Gorenstein (resp. $\bar{G}$-Auslander-Gorenstein, resp. $\bar{G}$-Auslander-regular and the number of $T(\ker\varphi)$ is invertible in $A$), it follows that $\varphi^*(A)$ is $H$-Gorenstein (resp. $H$-Auslander-Gorenstein, resp. $H$-Auslander-regular).
\end{proposition}

\begin{proof}
By Lemma \ref{Laurent-equivalence}, $A$ is $G$-Gorenstein (resp. $G$-Auslander-Gorenstein, resp. $G$-Auslander regular)  if and only if $A[t,t^{-1}]$ is $\bar{G}$-Gorenstein (resp. $\bar{G}$-Auslander-Gorenstein, resp. $\bar{G}$-Auslander-regular). The result now follows by Lemma \ref{Rees-ring-localization}, Theorem \ref{Auslander-cut-at-infinity-special} (with $(R,\hbar)=(R_F(A),t)$) and Theorem \ref{Auslander-regrading}.
\end{proof}

When $A$ is a $G$-graded algebra (over a field $\mathbb{K}$), each layer of a homo-filtration $F$ on $A$ is required to be a homogeneous subspace of $A$.  Then $G_F(A)$ and $R_F(A)$ are indeed $\bar{G}$-graded algebras.

\begin{proposition}\label{Cohen-Macaulay-filtration}
Let $A$ be a $G$-graded algebra and $\varphi:G\to H$ a group homomorphism. Let $F$ be a positive homo-filtration on $A$ such that all $F_nA$ are locally finite and $p(\supp R_F(A))\subseteq \N^r$ for some injective group homomorphism $p:\bar{G} \to \Z^r$. Then, if $G_F(A)$ is $\bar{G}$-Auslander-Gorenstein and $\bar{G}$-Cohen-Macaulay, it follows that  $\varphi^*(A)$ is $H$-Auslander-Gorenstein and  $H$-Cohen-Macaulay.
\end{proposition}

\begin{proof}
Assume $G_F(A)$ is $\bar{G}$-Auslander-Gorenstein and $\bar{G}$-Cohen-Macaulay. Clearly, $R_F(A)$ is well-supported and locally finite. It follows that $R_F(A)$ is $\bar{G}$-Auslander-Gorenstein and $\bar{G}$-Cohen-Macaulay by Theorem \ref{Cohen-Macaulary-cut-at-infinity-quotient}. Consider the map $\bar{\varphi}=\id_{\Z}\times \varphi: \bar{G} \to \bar{H}$. Then $\bar{\varphi}^*(R_F(A))$ is $\bar{H}$-Auslander-Gorenstein and $\bar{H}$-Cohen-Macaulay by Theorem \ref{Cohen-Macaulay-regrading}. Also,  Lemma \ref{Rees-ring-localization} tells us that $$\varphi^*(A)[t,t^{-1}] = \bar{\varphi}^*(A[t,t^{-1}]) \cong \bar{\varphi}^*(R_F(A)[t^{-1}])= \bar{\varphi}^*(R_F(A)) [t^{-1}]$$ as $\bar{H}$-graded algebras. Consequently, $\varphi^*(A)[t,t^{-1}]$ is is $\bar{H}$-Auslander-Gorenstein and $\bar{H}$-Cohen-Macaulay by Theorem \ref{Auslander-cut-at-infinity-special} (2) and Theorem \ref{Cohen-Macaulary-cut-at-infinity-localization}. Further, we have by Lemma \ref{gkdim-tensor} that $$\gkdim_{\varphi^*(A)[t,t^{-1}]}M[t,t^{-1}] = \gkdim_{\varphi^*(A)}M +1$$
for every module $M\in \Mod^H \varphi^*(A)$.
It follows that $\varphi^*(A)$ is $H$-Cohen-Macaulay by Lemma \ref{Laurent-equivalence}.
\end{proof}

\begin{example}
The first Wely algebra $A=A_1(\mathbb{K})$ is generated over $\mathbb{K}$ by two generators $x,\,y$ with one relation $xy-yx-1=0$. We introduce a $\Z$-grading on $A$ by putting $\deg(x)=1$ and $\deg(y)=-1$. Let  $F=\{\, F_nA\, \}_{n\in \Z}$ be the positive homo-filtration on $A$ given  by
$$
F_nA = \Span_{\mathbb{K}}\{\, x^iy^j\, |\, i\geq0,\, n\geq j \geq0\, \}.
$$
Clearly, $F_nA$ is infinite dimensional but locally finite for $n\geq0$; and $\supp R_F(A) = \N\, (1,0) +\N \, (1,-1)$. Also, it is not hard to see that $G_F(A) \cong \mathbb{K}[u,v]$ as $\Z^2$-graded algebras. Here the polynomial algebra $\mathbb{K}[u,v]$ is $\Z^2$-graded with $\deg(u)=(1,0)$ and $\deg(v)=(1,-1)$. Then, by Proposition \ref{Auslander-filtration} and Proposition \ref{Cohen-Macaulay-filtration}, one can conclude that $A$ is (ungraded) Auslander-regular and Cohen-Macaulay.
\end{example}

\begin{example}
In the work \cite{SZL2}, the authors classify out a class of Artin-Schelter regular algebras of dimension four with two generators. They are of the form $\mathcal{J}=\mathbb{K}\langle x,y\rangle/(f_1,f_2)$ with
\begin{eqnarray*}
\begin{split}
f_1 &= xy^2-2yxy+ y^2x,\\
f_2 &= x^3y-3x^2yx+3xyx^2-yx^3+(1-a)xyxy+ayx^2y \\
    &+(a-3)yxyx+(2-a)y^2x^2-by^2xy+by^3x+cy^4,
\end{split}
\end{eqnarray*}
where $a,b,c\in \mathbb{K}$, and with $\Z$-grading given by $\deg(x)=\deg(y)=1$.
Let $F=\{\, F_n\mathcal{J}\, \}_{n\in \Z}$ be the positive homo-filtration on $\mathcal{J}$ given by
$$
F_n\mathcal{J} = \text{the $\mathbb{K}$-span of all words that have at most $n$ appearances of $x$}.
$$
Clearly, $F_nA$ is infinite dimensional but locally finite for $n\geq0$; and $\supp R_F(A) = \N^2$. By the method of Gr\"{o}bner bases theory, it is not hard to see that $G_F(\mathcal{J}) \cong D(-2,-1)$ as $\Z^2$-graded algebras, where $$D(-2,-1) := \frac{\mathbb{K}\langle u,v\rangle} {(uv^2-2vuv+ v^2u,\, u^3v-3u^2vu+3uvu^2-vu^3 )}$$ is $\Z^2$-graded by $\deg(u)=(1,1)$ and $\deg(v)=(0,1)$. Then, by \cite[Thoerem C]{LPWZ}, Proposition \ref{Auslander-filtration} and Proposition \ref{Cohen-Macaulay-filtration}, one can conclude that $\mathcal{J}$ is (ungraded) Auslander-regular and Cohen-Macaulay.
\end{example}

\begin{remark}
In the work \cite{SZL1}, a method called ``Homogeneous PBW deformation'' is developed to construct new $\Z$-graded Artin-Schelter regular algebras from multi-graded old ones by adding to each relation a tail of the same $\Z$-degree but different multi-degree. It is the technique of multi-filtration to assure that this method actually preserves nice ring-theoretic and homological properties.  We want to mention that this job can be done more smoothly by the technique of homo-filtrations.

\end{remark}

\vskip7mm

\noindent{\it Acknowledgments.} G.-S. Zhou is supported by the NSFC (Grant No. 11601480); Y. Shen is supported by the NSFC (Grant No. 11626215) and Science Foundation of Zhejiang Sci-Tech University (Grant No. 16062066-Y);  D.-M. Lu is supported by the NSFC (Grant No. 11671351).

\end{document}